\documentclass{article}
\usepackage{graphicx} 
\usepackage{amsmath} 
\usepackage{amsthm}
\usepackage{amsfonts} 
\usepackage{mathtools} 
\usepackage{tikz-cd}
\usepackage{relsize} 

\usepackage{float} 

\usepackage[colorlinks=true,urlcolor=blue,
bookmarks=true,bookmarksopen=true,citecolor=blue]{hyperref}

\newtheorem{theorem}{Theorem}[section]

\newtheorem{lemma}[theorem]{Lemma}
\newtheorem{claim}[theorem]{Claim}
\newtheorem{proposition}[theorem]{Proposition}

\newtheorem{mainthm}{Theorem}
\theoremstyle{definition}
\newtheorem{definition}[theorem]{Definition}

\theoremstyle{remark}
\newtheorem*{remark*}{Remark}
\newtheorem{remark}[theorem]{Remark}

\newcommand{\wfhBold}[3]{
    \boldsymbol{\operatorname{WFH}}\left( #1,\,#2 \to #3 \right)
}

\newcommand{\wfhBoldSmallPar}[3]{
    \boldsymbol{\operatorname{WFH}}( #1,\,#2 \to #3)
}

\newcommand{\wfhPlain}[3]{
    \operatorname{WFH}\left( #1,\,#2 \to #3 \right)
}

\newcommand{\hmBold}[2]{
        \boldsymbol{\operatorname{HM}}\left(#1\,;#2\right)
}

\newcommand{\hmBoldSmallPar}[2]{
        \boldsymbol{\operatorname{HM}}(#1\,;#2)
}

\newcommand{\eps}{\varepsilon}

\begin{document}
\title{Poisson Bracket Invariants and Wrapped Floer Homology}
\author{Yaniv Ganor$^1$}
\footnotetext[1]{This work began when the author was a post-doc at the Technion. It was Partially supported by the Israel Science Foundation grant 1715/18, partially supported by Technion scholarship funds, and supported in part at the Technion by a fellowship from the Lady Davis Foundation}

\date{}

\maketitle

\begin{abstract}
    The Poisson bracket invariants, introduced by Buhovsky, Entov, and Polterovich and further studied by Entov and Polterovich, serve as invariants for quadruples of closed sets in symplectic manifolds. Their nonvanishing has significant implications for the existence of Hamiltonian chords between pairs of sets within the quadruple, with bounds on the time-length of these chords. In this work, we establish lower bounds on the Poisson bracket invariants for certain configurations arising in the completion of Liouville domains. These bounds are expressed in terms of the barcode of wrapped Floer homology. Our primary examples come from cotangent bundles of closed Riemannian manifolds, where the quadruple consists of two fibers over distinct points and two cosphere bundles of different radii, or a single cosphere bundle and the zero section.
\end{abstract}

\tableofcontents

\section{Introduction}

This paper examines the phenomenon of interlinking \cite{entov2017lagrangian, entov2022legendrian}, via the lens of Wrapped Floer homology, through the connection between interlinking and Poisson bracket invariants (introduced in \cite{buhovsky2012poisson, entov2017lagrangian}). Let us explain.

Let $(M,\omega)$ be a symplectic manifold. Let $(Y_0,Y_1)$ be a pair of disjoint sets: $Y_0\cap Y_1 = \emptyset$. We say that an autonomous Hamiltonian $H\colon M \to \mathbb{R}$, $\Delta$-separates $Y_0$ and $Y_1$ if 
\[ \inf_{Y_1} H - \sup_{Y_0} H = \Delta > 0. \]
Note that in the definition the order of the sets is important. The Hamiltonian $H$ is bigger on $Y_1$ than on $Y_0$.

The following phenomenon of interlinking was introduced in \cite{entov2017lagrangian}:
\begin{definition}
    Let $(X_0,X_1)$, $(Y_0,Y_1)$ be two pairs of disjoint sets: \[X_0\cap X_1 = Y_0 \cap Y_1 = \emptyset\]. We say that the pair $(Y_0,Y_1)$ \emph{autonomously $\kappa$-interlinks} the pair $(X_0,X_1)$, with $\kappa>0$, if for every complete autonomous Hamiltonian which $\Delta$-separates $Y_0$ and $Y_1$, the corresponding Hamiltonian flow which admits a chord from $X_0$ to $X_1$ of time-length $\le \kappa / \Delta$.
\end{definition}
Our convention for the Hamiltonian vector field is $\omega(X_H,\cdot) = -dH(\cdot)$.

In this work we study interlinking of certain configurations in completions of Liouville domains, through the relation between interlinking and Poisson bracket invariants, introduced by \cite{entov2017lagrangian}. We recall their definition in Definition \ref{def:introPoissonBracketDef} in the introduction and in Section \ref{sec:poissonBracketInvariants}.

Our method is to provide lower bounds on the Poisson bracket invariants via wrapped Floer homology, a Floer-type invariant associated to pairs of certain Lagrangians in Liouville domains. See \cite{alves2019dynamically} for a construction and relevant properties of wrapped Floer homology of pairs of Lagrangians.

The main example for which our method applies is that of cotangent fibers in a cotangent bundle.
Let $(N,g)$ be a closed Riemannian manifold. The metric $g$ on $M$ induces a bundle isomorphism between $TN$ and $T^*N$, hence $g$ defines a bundle metric on $T^*N$, through this isomorphism.

We turn to state our result on cotangent bundles.
\begin{mainthm}\label{thm:cotangentInterlinking}
Denote by $S_r^*N$ the cosphere bundle of radius $r$ inside $T^*N$. Pick $x,y\in N$, two points of distance $d$. Then, denoting by $T^*_q N$ the cotangent fiber over the corresponding point $q$, we have that for all $0<a<b$:
\begin{enumerate}
    \item 
    The pair $\left(S^*_a N, S^*_b N\right)$ autonomously $d\cdot \left(b-a\right)$-interlinks the pair $\left(T^*_x N, T^*_y N\right)$, and vice versa.
    \item 
    The pair $\left(0^{}_{T^*\!N}, S^*_a\right)$ autonomously $d\!\cdot\! a$-interlinks the pair $\left(T^*_x N, T^*_y N\right)$, and vice versa. Here, $0^{}_{T^*\!N}$ denotes the zero section.
\end{enumerate}
\end{mainthm}

\begin{remark}
    This work is inspired by \cite{entov2022legendrian}, where interlinking of certain pairs involving Lagrangian submanifolds are studied using methods of persistence in Legendrian contact homology. Due to fundamental issues in Legendrian contact homology, certain components of the theory required for their proof have not yet been proven in full generality, but rather, only for certain type of contact manifolds. Therefore the sets of examples covered by \cite{entov2022legendrian} and by this work are disjoint. In particular, cotangent bundles of \emph{compact} manifolds do not fall under the scope of \cite{entov2022legendrian}.
\end{remark}
\begin{remark}
    Similar, yet different, interlinking results were previously obtained in the case of the contangent bundle of a torus, $T^*T^n$, also via Poisson bracket invariants but with different methods (symplectic quasi-states).
    Both results are in the more general context of non-autonomous interlinking:
    
    Theorem 1.13 in \cite{buhovsky2012poisson} proves interlinking for a pair of cotangent fibers, and the pair comprising the zero section and the boundary of some neighborhood of the zero section, but without specifying the direction in which the Hamiltonian chord goes. 

    In his dissertation, \cite{rauchThesis}, Rauch proves a certain dichotomy between interlinking and a type of escaping of trajectories to infinity, for a pair of cotangent fibers, and the pair comprising the zero section and the boundary of a certain type of neighborhood of the zero section.
\end{remark}

Theorem \ref{thm:cotangentInterlinking} follows from a more general theorem on Lagrangians in Liouville domains.
Recall the definition of a \emph{Liouville domain}, $M=(Y,\omega,\lambda)$, where $Y$ is a manifold with boundary $\Sigma := \partial Y$, the $2$-form $\omega=d\lambda$ is an exact symplectic form and $\alpha^{}_M := \lambda\vert^{}_{\partial Y}$ is a contact form on $\partial Y$. Denote by $\xi^{}_\Sigma = \ker \alpha^{}_M$ the contact distribution. The form $\lambda$ is called the \emph{Liouville form}.
 
The \emph{Liouville vector field} $v$ is defined by $\mathcal{L}_v \omega = \omega$, or equivalently by $\omega(v,\cdot) = \lambda (\cdot)$.
Using the flow of the Liouville vector field, one identifies a collar neighborhood of the boundary $\partial Y$ with the symplectic manifold 
\[\left( \left(\varepsilon,1\right] \times \Sigma, \; d\left(r\alpha^{}_M\right)\right),\]
where $r$ is the coordinate on the interval. In fact, there exists a collar neighborhood of the form $\left( \left(0,1\right] \times \Sigma, \; d\left(r\alpha^{}_M\right)\right)$. Its complement in $M$ is called the \emph{core} of $M$, which we denote by $\operatorname{core}\left(M\right)$

One defines the \emph{completion} of $M$ to be the following Liouville manifold (without boundary) 
\[ 
    \widehat{M} = \left( Y\sqcup^{}_{\left\{1\right\}\times \Sigma} \left(\left[1,\infty \right) \times \Sigma\right),\; \omega^{}_{\widehat {M}},\; \lambda^{}_{\widehat {M}}\right),
\]
where the gluing is done along the aforementioned collar neighborhood, the symplectic form is given by $\omega^{}_{\widehat {M}}= d\lambda^{}_{\widehat {M}}$ and the form $\lambda^{}_{\widehat {M}}$ is defined by
\[
    \begin{aligned} 
        &\left.\lambda^{}_{\widehat {M}}\right\vert^{}_{M} := \lambda \\
        &\left.\lambda^{}_{\widehat {M}}\right\vert^{}_{\left[1,\infty \right)\times \Sigma} := d\left(r\alpha^{}_M\right)
    \end{aligned}
\]
\begin{remark}
    To simplify the notation we would often write $M$ and $\widehat{M}$ instead of $Y$ and $\widehat{Y}$, and in similar spirit, by $\partial M$ we mean $\partial Y=\Sigma$.
\end{remark}

We consider Lagrangians with boundary $L\subset M,\enskip\partial L\subset \partial M,$ which are \emph{exact}, namely, $\left.\lambda\right\vert^{}_{L} = df$ for some $f\colon L \to \mathbb{R}$. We say that an exact Lagrangian  $L\subset M$ is \emph{asymptotically conical} if the following two conditions hold:
\begin{enumerate}
    \item The boundary $\Lambda := \partial L$ is a Legendrian submanifold of $\left(\Sigma, \xi^{}_\Sigma \right)$
    \item There exists $\varepsilon > 0$ such that $L$ restricted to the collar neighborhood $\left[1-\varepsilon, 1\right]\times \Sigma$ is of the form $\left[1-\varepsilon, 1\right] \times \Lambda$.
\end{enumerate}
Given an asymptotically conical Lagrangian $L \subset M$, we can complete it to $\widehat {L} \subset \widehat {M}$, a Lagrangian without boundary in the completion of $M$, by attaching a conical end, namely,
\[ 
    \widehat{L} = L\sqcup^{}_{\left\{1\right\}\times \Lambda} \left(\left[1,\infty \right) \times \Lambda\right).
\]

Now, consider the following setup: $M$ is a Liouville domain and $L,L^\prime \subset M$ are two exact Lagrangians, which are asymptotically conical. Choose $f,g$ such that $\left.\lambda\right\vert^{}_{L} = df$, $\left.\lambda\right\vert^{}_{L^\prime} = dg$ and such that $f$ and $g$ vanish in a neighborhood of $\partial M$.

To such pairs of Lagrangians one can associate the \emph{action filtered wrapped Floer homology}, a family of vector spaces $HW^a\left(M,L\to L^\prime\right)$, parametrized by $a\in \mathbb{R}$, depending only of $M, L, L^\prime, f$ and $g$. See Section \ref{sec:wrappedFloerHomology} for the definition of the action functional and of wrapped Floer homology. 
The wrapped Floer homology of action sublevels is a \emph{persistence module}, (see Section \ref{sec:perModules} on persistence modules.) Namely, for every $a\le b$ there exists a morphism 
\[
    \iota_{a\to b} \colon HW^a\left(M,L\to L^\prime\right) \to HW^b\left(M,L\to L^\prime\right),\]
satisfying 
\begin{itemize}
    \item $\iota_{a\to a} = \operatorname{Id}$,
    \item $\iota_{b\to c} \circ \iota_{a\to b}  = \iota_{a\to c}$.
\end{itemize}
Persistence modules admit a normal form which is classified by collections of intervals, called \emph{barcodes}, see Theorem \ref{thm:perModNormalForm}, cited in Section \ref{sec:perModules}. In our setting the barcode is a multiset of intervals either of the form $(a,b]$, for $a<b$, or $(a,\infty)$.
When thought of as a persistence module, we denote the wrapped Floer homology by $\wfhBold{M}{L}{L^\prime}$, with $\wfhPlain{M}{L}{L^\prime}_a = HW^a\left(M,L\to L^\prime\right)$.

Our main theorem is formulated in terms of the Poisson bracket invariant of quadruples, introduced in \cite{buhovsky2012poisson,entov2017lagrangian}, which we briefly recall in what follows, and elaborate on in section \ref{sec:poissonBracketInvariants}.

Four compact subsets $X_0,X_1,Y_0,Y_1$ of $M$ will be called an \emph{admissible quadruple} if 
\[
    X_0 \cap X_1 = Y_0 \cap Y_1 = \emptyset.
\]
Given an admissible quadruple define the following set of pairs of compactly supported smooth functions:
\[\mathcal{F}_M\left(X_0,X_1,Y_0,Y_1\right)= \left\lbrace \left(F,G\right) \in C^\infty_c\left(M\right) \times C^\infty_c\left(M\right)  \,\middle\vert\,
\begin{aligned}
    F\vert^{}_{X_0} \le 0, && F\vert^{}_{X_1} \ge 1 \\
    G\vert^{}_{Y_0} \le 0, && G\vert^{}_{Y_1} \ge 1 \\
\end{aligned}
\right\rbrace.
\]

The Poisson bracket invariant of an admissible quadruple is defined as:
\begin{definition}\label{def:introPoissonBracketDef}
\[
    \operatorname{pb}^+_M\left(X_0,X_1,Y_0,Y_1\right) := \inf_{\mathcal{F}_M\left(X_0,X_1,Y_0,Y_1\right)} \max_M \left\lbrace F,G \right\rbrace.
\] 
\end{definition}

Given a Liouville domain and $a>0$, we define the following subset of its completion $\widehat{M}$: 
\[
    M_a:= \operatorname{core}\left(M\right) \sqcup \left(0,a\right] \times \Sigma.
\]
Note that $M_1 = M$, and if $a<1$ then $M_a \subset M$. Equipped with the restricted Liouville form, $M_a$ is also a Liouville domain by itself.
We denote by $M_{[a,b]}$ the subset of $\widehat M$ given by $\left[a,b\right] \times \Sigma$.

Our main theorem now reads: 
Let $L,L^\prime \subset M$ be two exact Lagrangians in a Liouville domain $M$, which are asymptotically conical, and fix a choice of $f,g$ such that $\left.\lambda\right\vert^{}_{L} = df$, $\left.\lambda\right\vert^{}_{L^\prime} = dg$ and such that $f$ and $g$ vanish in a neighborhood $\partial M$.
\begin{mainthm}\label{thm:mainTheorem}
    Assume that the barcode of $\wfhBold{M}{L}{L^\prime}$ contains a bar of the form $\left (\mu, C\mu\right]$ with $\mu > 0$, $C>1$. (For a bar of the form $\left(\mu, \infty\right)$ we set $C=\infty$) Then:
    \begin{enumerate}
        \item\label{itm:1stInMainTheorem} For every $0<a<b$ with $\frac {b}{a} \le C$, if $\widehat{L}, \widehat{L^\prime}$ are conical in $M_{[a,b]}$ then
                \[
                \operatorname{pb}_{\widehat {M}}^+\left(\widehat{L}\cap  M_{[a,b]}, \widehat{L^\prime} \cap  M_{[a,b]},  \partial M_a, \partial M_b \right) \ge \frac {1}{\mu\left(b-a\right)}.
                \]
            \item If moreover, $C=\infty$ and $\widehat{L}, \widehat{L^\prime}$ are conical in $\widehat{M}\setminus \operatorname{core}\left(M\right)$ then:
                \[
                \operatorname{pb}_{\widehat {M}}^+\left(\widehat{L}\cap  M_{b}, \,\widehat{L^\prime}\cap  M_{b}, \,\operatorname{core}\left(M\right), \,\partial M_b \right) \ge \frac {1}{\mu b}.
                \]
    \end{enumerate}
\end{mainthm}
\begin{remark}
    Note that as $\mu$ tends to $0$, both bounds in Items 1 and 2, tend to $+\infty$. Also note that in particular, in out setting, there can be no bar with a left endpoint at zero, as $\operatorname{pb}^+$ is always finite.
\end{remark}
The relation with interlinking comes from the following property of Poisson bracket invariants.
\begin{theorem}[{\cite[Theorem 1.11]{entov2017lagrangian}}]\label{thm:pb4AndInterlinking}
    Suppose $(X_0,X_1,Y_0,Y_1)$ is an admissible quadruple with $\operatorname{pb}_M^+\left(X_0,X_1,Y_0,Y_1\right) = \frac{1}{\kappa} > 0$. Then the pair $(Y_0,Y_1)$ autonomously $\kappa$-interlinks the pair $(X_0,X_1)$.\qed
\end{theorem}

\begin{remark}
    Similarly to the notation $M_{[a,b]}$ for the set given by $\left[a,b\right] \times \Sigma$, we denote by $\widehat{L}_{[a,b]}$ the intersection $\widehat{L}\cap M_{[a,b]}$. Then in fact, in Theorem \ref{thm:mainTheorem}, Item \ref{itm:1stInMainTheorem}, we have for all $\delta > 0$:
            \[
                \operatorname{pb}_{M_{[a-\delta,b+\delta]}}^+\left(\widehat{L}_{[a,b]},\widehat{L^\prime}_{[a,b]},  \partial M_a,  \partial M_b \right) \ge \frac {1}{\mu\left(b-a\right)}.
            \]
    Note that here the ambient manifold in $\operatorname{pb}^+$ is not $\widehat{M}$.
    The proof follows from cutoff arguments similar to the proof of Claim \ref{clm:pbConvenientForm}.
    In light of the relation to interlinking (Theorem \ref{thm:pb4AndInterlinking}), it follows that the results on the interlinking continue to hold whenever the configuration of $\left(\widehat{L}_{[a,b]},\widehat{L^\prime}_{[a,b]},  \partial M_a,  \partial M_b \right)$ embeds into a symplectic manifold $W$, via a codimension 0 embedding of $M_{[a-\delta, b+\delta]}$ into $W$. This is due to the monotonicity property of $\operatorname{pb}^+$, \cite[Section 4]{entov2017lagrangian}.
\end{remark}

\begin{remark}
By the anti-symmetry property of $\operatorname{pb}_4^+$, we also have that \\
$\operatorname{pb}_M^+\left(Y_1,Y_0,X_0,X_1\right) = \frac{1}{\kappa}$, and hence by the theorem also the pair $(X_0,X_1)$ autonomously $\kappa$-interlinks the pair $(Y_1,Y_0)$. (Note the reversed order of the sets in the second pair).
\end{remark}

Theorem \ref{thm:cotangentInterlinking} is thus obtained from Theorem \ref{thm:mainTheorem} by a computation of wrapped Floer homology for cotangent fibers.

\subsection{Structure of the Paper}
In Section \ref{sec:poissonBracketInvariants} we review the Poisson bracket invariant, and in Section \ref{sec:perModules} we give a short review of persistence modules, following \cite{entov2022legendrian}. 

Section \ref{sec:wrappedFloerHomology} begins by reviewing wrapped Floer homology for a pair of exact Lagrangians, and then constructs the tools that we need for the proof of our main theorems. In particular, in subsection \ref{ssec:continuationMovingBdry} we construct a continuation map for Lagrangian isotopies, which plays crucial role in our proof. In subsection \ref{ssec:energyAndC0} we provide $C^0$ and energy bound required for establishing the required properties of that continuation map.

Section \ref{sec:mainTheoremProof} is dedicated to the proof of Theorem \ref{thm:mainTheorem}, and Section \ref{sec:ctgtBundles} to the proof of Theorem \ref{thm:cotangentInterlinking} about interlinking of configurations in cotangent bundles. Section \ref{sec:discussion} is a short discussion of further research directions.

\subsection{Acknowledgements}
This work began when the author was a post-doc at the Technion, hosted by Michael Entov. I would like to thank Michael Entov for suggesting that I look into this project and for the useful discussions. I also would like to thank Leonid Polterovich for useful discussion, and thank Itamar Rauch for discussions both on his his dissertation results and on this work. Lastly, I would like to thank the math department at the Technion for three pleasant years of post-doc.

\section{Poisson Bracket Invariant}\label{sec:poissonBracketInvariants}

Recall that four sets $X_0,X_1,Y_0,Y_1$ will be called an \emph{admissible quadruple} if 
\[
    X_0 \cap X_1 = Y_0 \cap Y_1 = \emptyset.
\]
Given an admissible quadruple define the following set of pairs of compactly supported smooth functions:
\[\mathcal{F}_M\left(X_0,X_1,Y_0,Y_1\right)= \left\lbrace \left(F,G\right) \in C^\infty_c\left(M\right) \times C^\infty_c\left(M\right)  \,\middle\vert\,
\begin{aligned}
    F\vert^{}_{X_0} \le 0, && F\vert^{}_{X_1} \ge 1 \\
    G\vert^{}_{Y_0} \le 0, && G\vert^{}_{Y_1} \ge 1 \\
\end{aligned}
\right\rbrace.
\]
Define the Poisson bracket invariant of an admissible quadruple to be:
\begin{equation}\label{eqn:pb4Definition}
    \operatorname{pb}^+_M\left(X_0,X_1,Y_0,Y_1\right) := \inf_{\mathcal{F}_M\left(X_0,X_1,Y_0,Y_1\right)} \max_M \left\lbrace F,G \right\rbrace.
\end{equation}

For a subest $X\subset M$ and a function $f\in C^\infty(M)$ we write $f\vert^{}_{\mathcal{O}p\left(X\right)}\equiv C$ to mean ``there exists some neighborhood of $X$ such that $f$ restricted to that neighborhood is a constant function equal to $C$''.

We now define a subest of $\mathcal{F}_M\left(X_0,X_1,Y_0,Y_1\right)$ of functions that are constant in neighborhoods of the corresponding sets, namely:
\begin{multline*}    
\mathcal{G}_M\left(X_0,X_1,Y_0,Y_1\right) = \\
=\left\lbrace \left(F,G\right) \in C^\infty_c\left(M\right) \times C^\infty_c\left(M\right)  \,\middle\vert\,
\begin{aligned}
    F\vert^{}_{\mathcal{O}p\left(X_0\right)} \equiv 0, && G\vert^{}_{\mathcal{O}p\left(Y_0\right)} \equiv 0 \\
    F\vert^{}_{\mathcal{O}p\left(X_1\right)} \equiv 1, && G\vert^{}_{\mathcal{O}p\left(Y_1\right)} \equiv 1 \\
\end{aligned}
\right\rbrace.
\end{multline*}
Taking the infimum over $\mathcal{G}_M\left(X_0,X_1,Y_0,Y_1\right)$ Formula (\ref{eqn:pb4Definition}) yields an equivalent definition. This is proven in \cite{buhovsky2012poisson} for a similar invariant with the $C^0$ norm instead of $\max$, and the proof is analogous. The set $\mathcal{G}_M\left(X_0,X_1,Y_0,Y_1\right)$ is more convenient for our scheme of proving a lower bound.

We recall the following properties of $\operatorname{pb}^+_M$. (See \cite{entov2017lagrangian}.):
\begin{description}
    \item[Semi-continuity]: Suppuse that a secquence $X^{(j)}_0$, $X^{(j)}_1$, $Y^{(j)}_0$, $Y^{(j)}_1$, $j\in \mathbb{N}$ of ordered collections of sets converges in the Hausdorff distance to a quadruple $X_0$, $X_1$, $Y_0$, $Y_1$. Then:
    \[
        \operatorname{pb}^+_M\left(X_0, X_1, Y_0, Y_1\right) \ge \limsup_{j\to \infty} \operatorname{pb}^+_M\left(X^{(j)}_0, X^{{(j)}}_1, Y^{(j)}_0, Y^{(j)}_1\right).
    \]
    \item[Anti-symmetry]: 
    \[
        \operatorname{pb}^+_M\left(X_0, X_1, Y_0, Y_1\right) = \operatorname{pb}^+_M\left(Y_1, Y_0, X_0, X_1\right).
    \]
    Note the reversal in the order of the sets $Y_1$ and $Y_0$.
\end{description}

\section{Persistence Modules}\label{sec:perModules}
In this section we recall basic facts about persistence modules, following the summary given in \cite{entov2022legendrian}. For a more detailed introduction to persistence modules see e.g. 
\cite{persistence1_edelsbrunner2022computational}, \cite{persistence2_chazal2016structure}, \cite{persistence3_oudot2017persistence} or \cite{persistence4_polterovich2020topological}.
We work over the base field $\mathbb{Z}/2$, since this case is enough for our purposes.

\begin{definition}
    A persistence module is functor from the poset category of $\left(\mathbb R, \le \right)$ to $\operatorname{Vect}^{\operatorname{fin}}_{\mathbb{Z}/2}$, the category of finite dimensional vector spaces over $\mathbb{Z}/2$. Explicitly, a persistence module $\mathbf{V}$ is given by a pair
    \[ \left(V=\left\lbrace V_t\right\rbrace_{t\in\mathbb R}, \pi = \left\lbrace \pi_{s,t}\right\rbrace_{s,t\in\mathbb R,\, s\le t}\right),\]
    where $V_t$ is a finite dimensional $\mathbb{Z}/2$-vector space for all $t\in\mathbb{R}$, and $\pi_{s,t} \colon V_s \to V_t$ are linear maps (called \emph{the persistence maps}) satisfying
    \begin{enumerate}
        \item \emph{(Persistence):} $\pi_{t,t}=\operatorname{Id}$ for all $t\in \mathbb{R}$, and $\pi_{s,t} = \pi_{r,t} \,\circ\, \pi_{s,r}$ for all $s\le r \le t$, $s,r,t \in \mathbb{R}$.
    \end{enumerate}
\end{definition}
\begin{definition}
    A persistence module is called \emph{a persistence module of finite type} if moreover it satisfies:
    \begin{enumerate}\setcounter{enumi}{1} 
        \item \emph{(Discrete spectrum and Semi-continuity):} There exists a (finite or countable) discrete closed set of points 
        \[\operatorname{spec}(\mathbf{V}) = \left\lbrace l_{min}(\mathbf{V}) := t_0 < t_1 < t_2 < \ldots < +\infty \right\rbrace \subset \mathbb{R},\]
        called the spectrum of $\mathbf V$, so that
        \begin{itemize}
            \item for any $r \in \mathbb{R}\setminus \operatorname{spec}(\mathbf{V})$ there exists a neighborhood $U$ of $r$ such that $\pi_{s,t}$ is an isomorphism for all $s,t \in U, \,s \le t$;
            \item for any $r \in \operatorname{spec}(\mathbf{V})$ there exists $\varepsilon > 0$ such that $\pi_{s,t}$ is an isomorphism of vector spaces for all $s,t \in (r-\varepsilon,r]$.
        \end{itemize}
        \item \emph{(Semi-bounded support):} For the smallest point $l_{min}(\mathbf{V}) := t_0$ of $\operatorname{spec}(\mathbf{V})$ one has $l_{min}(\mathbf{V}) > -\infty$, and $V_t = 0$ for all $t \le l_{min}(\mathbf{V})$.
    \end{enumerate}  
\end{definition}
    The notions of a persistence submodule, the direct sum of persistence modules and a morphism/isomorphism between persistence modules are defined in a straightforward manner.

Let $\mathbb{J}$ be an interval of the form $(a,b]$ for $-\infty<a<b<\infty$ or of the form $(a,\infty)$, for $a\in\mathbb{R}$. The \emph{interval persistence module},

\[\mathbf{Q}(\mathbb{J}) =(Q(\mathbb{J}), \pi) =  \left(\left\lbrace Q(\mathbb{J})_t\right\rbrace_{t\in\mathbb{R}}, \left\lbrace\pi_{s,t}\right\rbrace_{s,t\in\mathbb{R},\,s\le t}\right), \]
is defined as follows: $Q(\mathbb{J})_t = \mathbb{Z}/2$ for $t\in\mathbb{J}$ and $Q(\mathbb{J})_t=0$ otherwise. The persistence maps $\pi_{s,t}$ are the identity map for $s,t\in\mathbb{J}$ and the zero map if $s\not\in\mathbb{J}$ or $t\not\in\mathbb{J}$.

The following structure theorem for persistence modules can be found in \cite{persistence5_zomorodian2004computing}, \cite[Thm. 2.7,2.8]{persistence2_chazal2016structure}, \cite{persistence6_crawley2015decomposition}. It has various versions which have appeared prior to the appearance of persistence modules, see the references provided in Section 3 of \cite {entov2022legendrian}.

\begin{theorem}[Structure theorem for persistence modules]\label{thm:perModNormalForm}

    For every persistence module of finite type $\mathbf{V} = (V,\pi)$ over $\mathbb{R}$ there exists a unique (finite or countable) collection of intervals $\mathbb{J}_j \subseteq \mathbb{R}$ where the intervals may not be distinct but each interval appears in the collection only finitely many times – so that $\mathbf{V} = (V,\pi)$ is isomorphic to 
    \[
    \pushQED{\qed} 
    \bigoplus_j \mathbf{Q}(\mathbb{J}_j). \qedhere 
    \popQED\]
\end{theorem}

The collection $\left\lbrace\mathbb{J}_j\right\rbrace$ of the intervals is called the \emph{barcode} of $\mathbf{V}$. The intervals $\mathbb{J}_j$ themselves are called the \emph{bars} of $\mathbf{V}$. Note that the same bar may appear in the barcode several (but finitely many) times (this number of times is called the multiplicity of the bar) – in other words, a barcode is a multiset of bars. The barcode of the trivial persistence module is empty.

The main persistence module we are concerned with in this work is wrapped Floer homology. The modules and persistence maps are defined in Section \ref{sec:wrappedFloerHomology}: \[\wfhBold{M}{L}{L^\prime}:=\left(\wfhPlain{M}{L}{L^\prime}, \pi\right),\] with 
\[\wfhPlain{M}{L}{L^\prime}_a = HW^{a}\left(M,L\to L^\prime\right),\ \ \text{and}\ \  \pi_{a,b} = \iota_{a\to b}.\]

The following lemma will be crucial in our proof of Theorem \ref{thm:mainTheorem}.

\begin{lemma}\label{lem:pmNonZeroElement}
    Let $\mathbf{V}$ be a persistence module and assume that the barcode of $\mathbf{V}$ contains a bar of the form $(a,b]$, or of the form $(a,\infty)$. In the latter case, set $b=\infty$ in what follows.
    Then, for every $\delta>0$ there exists an element $x\in V_{a+\delta}$ such that:

    \begin{itemize}
        \item $\pi_{a + \delta,\, s}\,x \neq 0$ for all $s$ such that $a + \delta\le s\le b$.
        \item For all $s$ such that $a+\delta \le s \le b$ and for all $\sigma\le a$: 
        \[ \pi_{a+\delta,\, s}\,x \not\in \operatorname{im}\pi_{\sigma,\, s}.\]  
    \end{itemize}
\end{lemma}
\begin{proof}
    $\mathbf{V}$ is isomorphic to a direct sum of interval modules, $\bigoplus_j \mathbf{Q}(\mathbb{J}_\mathbf{j})$.

    Since in the direct sum the persistence maps are direct sums of persistence maps of each interval module, it is enough to prove the lemma for the interval module $\mathbf{Q}(a,b]$, for which the lemma holds by its definition. Indeed, pick $x$ to be any non-zero element in $Q(a,b]_{a+\delta}$. Since $\pi_{a + \delta,\, s}=\operatorname{Id}$ for $a + \delta\le s\le b$, and $\pi_{\sigma,\, s} = 0$ for $s\le a$ the desired properties hold.
\end{proof}

\begin{remark}
    Let $\mathbf{V}$ be a persistence module. Let $c \in \mathbb{R}$. Define a new persistence module $\mathbf{V}^{[+c]}$ by adding $c$ to all indices of $V_t$ and $\pi_{s,t}$. In particular, $V^{[+c]}_t := V_{t+c}$. The barcode of $\mathbf{V}^{[+c]}$ is the barcode of $\mathbf{V}$ shifted by $c$ to the left if $c>0$ and shifted by $\vert c \vert$ to the right, if $c<0$.
\end{remark}

\section{Wrapped Floer Homology}\label{sec:wrappedFloerHomology}
\subsection{Introduction and Definitions}
We review the construction and properties of wrapped Floer homology of a pair of exact Lagrangians, following \cite{alves2019dynamically}. We restrict the discussion to the case where the Lagrangians are disjoint, which simplifies some aspects of the presentation. This generality is sufficient for our use case.

For two asymptotically conical exact Lagrangians $L$ and $L^\prime$ in $M$ denote by
\[
\mathcal{P}_{L\to L^\prime} = \left\lbrace \gamma \colon [0,1] \to \widehat{M} \,\middle\vert\, \gamma(0)\in \widehat{L}, \,\gamma(1)\in \widehat{L^\prime} \right\rbrace,
\]
the space of smooth paths from $\widehat{L}$ to $\widehat{L^\prime}$.

Denote by $X_{\alpha^{}_M}$ the \emph{Reeb vector field} on the boundary $\left(
\Sigma=\partial M, \xi^{}_M = \ker \alpha^{}_M \right)$, and denote the Reeb flow by $\phi^t_{X_{\alpha_M}}$. Recall that $X_{\alpha^{}_M}$ is defined by $X_{\alpha^{}_M} \in \ker d\alpha$ and $\alpha\left(X_{\alpha^{}_M}\right)=1$ . A \emph{Reeb chord of length $T$} from $\Lambda = \partial L$ to $\Lambda^\prime = \partial L^\prime$ is a map $\gamma \colon [0,T] \to \Sigma$, such that 
\begin{align*}
    \dot \gamma (t) &= X_{\alpha^{}_M} \left(\gamma(t))\right), \\
    \gamma(0) &\in \Lambda, \\
    \gamma(1) &\in \Lambda^\prime.
\end{align*}
A Reeb chord of length $T$ from $\Lambda$ to $\Lambda^\prime$ is called \emph{transverse} if the following two subspaces of $T_{\gamma\left(T\right)} \Sigma$, intersect at a single point:  $T^{}_{\gamma\left(T\right)} \left(\phi^T_{X_{\alpha^{}_M}}\Lambda\right)$ and $T^{}_{\gamma\left(T\right)} \Lambda^\prime$.

The \emph{spectrum} of the triple $(\alpha^{}_M, \Lambda \to \Lambda^\prime)$, is the set of all lengths of Reeb chords from $\Lambda$ to $\Lambda^\prime$ and is denoted by $\mathcal{S}\left(\Sigma, \Lambda \to \Lambda^\prime\right)$. It is known to be a nowhere dense set in $[0,\infty)$.

Given a pair of Legendrian submanifolds $\Lambda$ and $\Lambda^\prime$ in $\left(\Sigma, \xi\right)$ we say that a contact form $\alpha^{}_M$ is \emph{regular}, if all Reeb chords from $\Lambda$ to $\Lambda^\prime$ are transverse. Regularity implies that the spectrum  $\mathcal{S}\left(\Sigma, \Lambda \to \Lambda^\prime\right)$ is discrete.
We say that $\left(M, L \to L^\prime\right)$ is \emph{regular} if $\left(\lambda\vert^{}_\Sigma, \Lambda \to \Lambda^\prime\right)$ is regular and $L$ and $L^\prime$ intersect transversely. Note that we restrict ourselves to the case where $L\cap L^\prime = \emptyset$, so the latter automatically holds.

It is a known fact that for any contact form $\alpha$ on $\left(\Sigma, \xi\right)$, there exists arbitrarily small $f\colon \Sigma \to \mathbb{R}$ such that $(1+f)\alpha$ is regular. See \cite[Lemma 1.4]{bourgeois2002morse}.

From here on we assume that $\left(M, L \to L^\prime\right)$ is regular. \\ A Hamiltonian $H_t\colon \widehat{M}\times [0,1] \to \mathbb{R}$ is called \emph{admissible} if 
\begin{enumerate}
    \item $H_t<0$ on $M\subset\widehat{M}$, 
    \item and there exist constants $\mu>0$ and $b<-\mu$ such that $H(r,x)= h(r)=\mu r + b$ on $[1,\infty) \times \partial M$, where $x$ denotes the coordinate on $\partial M$.
\end{enumerate}
If $H$ is an admissible Hamiltonian satisfying the second condition with a given $\mu$, we say that $H$ is \emph{admissible with slope $\mu$ at infinity}.

Fixing functions $f$ and $g$ on $\widehat{L}$ and $\widehat{L}^\prime$ respectively, such that $df = \lambda\vert^{}_{\widehat{L}}$ and $dg = \lambda\vert^{}_{\widehat{L}^\prime}$, We define the action functional $\mathcal{A}_H^{L\to L^\prime} = \mathcal{A}_H \colon \mathcal{P}_{L\to L^\prime} \to \mathbb{R}$ by
\[
\mathcal{A}_H \left(\gamma\right) = f\left(\gamma\left(0\right)\right) - g\left(\gamma\left(1\right)\right) + \intop_0^1 \gamma^*\lambda - \intop_0^1 H\left(\gamma\left(t\right)\right)dt.
\]
The critical points of the action functional are Hamiltonian chords of time length $1$ going from $\widehat{L}$ to $\widehat{L^\prime}$. We denote the set of all these chords by $\mathcal{T}_{L\to L^\prime}(H)$.

A Hamiltonian $H$ is called \emph{non-degenerate} if all elements of $\mathcal{T}_{L\to L^\prime}(H)$ are non-degenerate, namely if $\phi_H^1(\widehat{L})$ is trasnverse to $\widehat{L^\prime}$. 

Note that if $\left(M,L\to L^\prime\right)$ is regular, then a Hamiltonian $H$ with slope $\mu\not\in \mathcal{S}\left(\partial M, \partial L\to \partial L^\prime\right)$ is non-degenerate. Such slopes are called \emph{non-characteristic}. Moreover, since we restrict our discussion to disjoint Lagrangians, it is enough to require that $\left(\alpha^{}_M, \Lambda, \Lambda^\prime \right)$ is regular.
We denote by $\mathcal{H}_{\operatorname{reg}}$ the set of all admissible Hamiltonians which are non-degenerate for $L\to L^\prime$. 

An almost complex structure $J$ is said to be \emph{compatible} with $\omega$ if $\omega(J\cdot,\cdot) > 0$, and $\omega(J\cdot,J\cdot) = \omega(\cdot,\cdot)$, i.e $g_J(\cdot,\cdot) := \omega(J\cdot,\cdot)$ is a $J$-invariant Riemannian metric. Note that the convention $\omega(J\cdot,\cdot) > 0$ is reversed with respect to the convention used in other papers of the author. We use this convention here for compatibility with \cite{alves2019dynamically}.
An almost complex structure on $\left((0,\infty)\times \partial M), \lambda = rd\alpha^{}_M\right)$ is called \emph{cylindrical} if it preserves $\xi_M = \ker \alpha^{}_M$, if $J\vert^{}_{\xi_M}$ is independent of $r$ and compatible with $d\left(r\alpha^{}_M\right)$, and if $JX_{\alpha^{}_M} = r\partial_r$. An almost complex sturcture on $\widehat{M}$ is \emph{asymptotically cylindrical} if it is cylindrical on $[r_0,\infty)\times \partial M$ for some $r_0 > 1$.

We interpret the negative gradient flow of the action functional as Floer strips, and define the Floer complex of an admissible Hamiltonian $H$ and an asymptotically cylindrical almost complex structure, via the count of solutions to the Floer equation, namely, we consider the moduli space of Floer solutions of Fredholm index $1$, connecting two time-$1$ Hamiltonian chords $\gamma_-$ and $\gamma_+$ going from $\widehat{L}$ to $\widehat{L^\prime}$. We denote it by $\mathcal{M}^1\left(\gamma_-,\gamma_+, H, J\right)$, and its elements are the index $1$ solutions satisfying:
\begin{equation}
    \begin{aligned}
        &u\colon \mathbb{R}\times [0,1] \to \widehat{M}, \\
        &\overline{\partial}^{}_{J,H}u := \partial_s u + J_t\left( \partial_t u - X_{H_t}\left(u\right) \right) = 0, \\
        &\lim_{s\to\pm \infty} u(s,t) = \gamma_{\pm}(t), \\
        &u(s,0)\in \widehat{L} \text{ and }  u(s,1)\in \widehat{L^\prime}.
    \end{aligned}  
\end{equation}
We denote by $\widetilde{\mathcal{M}}^1\left(\gamma_-,\gamma_+, H, J\right) := \mathcal{M}^1\left(\gamma_-,\gamma_+, H, J\right) /\, \mathbb{R}$ the quotient by the $\mathbb{R}$-action of translation in the $s$-coordinate.

For a non-degenerate $H_t$ and a generic $J_t$, $t\in [0,1]$, one defines the \emph{wrapped Floer complex} as follows:
\[
CW\left(H,J_t, L\to L^\prime\right) = \bigoplus_{\gamma \in \operatorname{Crit}\left(\mathcal{A}_H\right)} \mathbb{Z}/2 \cdot \gamma,
\]
with a differential $\partial$ that is defined by
\[
    \partial(x) = \sum_{x^\prime \in \operatorname{Crit}\left(\mathcal{A}_H\right)} \#^{}_2\widetilde{\mathcal{M}}^1\left(\gamma_-,\gamma_+, H, J\right) \cdot x^\prime,
\]
where $\#^{}_2$ stands for the count mod $2$. For a generic $J_t$, the moduli spaces are cut transversally, the differential is well defined and satisfies $\partial^2 = 0$. A crucial feature of the differential  is that it does not increase action, this induces a filtration of the complex by subcomplexes given by action sublevels.

The wrapped Floer homology of $H$ is defined as the homology of the complex defined above,
\[
HW\left(H, L\to L^\prime\right) := H\left(CW\left(H, J_t, L\to L^\prime\right)\right).
\]

We suppress the dependence of $H$ on $t$ in the notation. Moreover, using the method of continuation maps, one can show that the homology does not depend on a generic choice of $J_t$, hence we suppress $J_t$ from the notation, here and in what follows.

Given two non-degenerate admissibble Hamiltonians $H_- < H_+$ one defines the \emph{continuation map} $\iota^{H_-,H_+} \colon CW\left(H_-, L\to L^\prime\right) \to CW\left(H_+, L\to L^\prime\right)$  as follows:
Choose a generic increasing homotopy through admissible Hamiltonians, $\left(H_s\right)_{s\in\mathbb{R}}$, $\partial_s H_s \ge 0$ with $H_s = H_\pm$ near $\pm \infty$, respectively. Then $\iota^{H_-,H_+}$ is defined by the mod $2$ count of elements in the  moduli space of Fredholm index $0$ solutions sastisfying
\begin{equation}
    \begin{aligned}
        &u\colon \mathbb{R}\times [0,1] \to \widehat{M}, \\
        &\overline{\partial}_{J,H_s}u := \partial_s u + J_t\left( \partial_t u - X_{H_s}\left(u\right) \right) = 0, \\
        &\lim_{s\to\pm \infty} u(s,t) = \gamma_{\pm}(t), \\
        &u(s,0)\in \widehat{L} \text{ and }  u(s,1)\in \widehat{L^\prime},
    \end{aligned}  
\end{equation}
for a generic choice of $J_t$. The genericity of $H_s$ and $J_t$, ensures the moduli spaces are cut transversally.

The continuation map is a chain map, hence descends to a map between the wrapped Floer homologies:
\[
    \iota^{H_-,H_+} \colon CW\left(H_-, L\to L^\prime\right) \to CW\left(H_+, L\to L^\prime\right).
\]
The continuation maps also do not increase action.

The wrapped Floer homology of $M$ and the Lagrangian $L$ and $L^\prime$, denoted $HW(M,L\to L^\prime)$, is defined as the direct limit of the Floer homologies over all $H\in \mathcal{H}_{\operatorname{reg}}(M,L\to L^\prime)$:
\[
HW(M,L\to L^\prime) := \varinjlim_{H} HW(H,L\to L^\prime). 
\]

\subsection{Action Filtration}
One also defines an action-filtered version of wrapped Floer homology, as follows.
For $a\in [0,\infty)$, let 
\[CW^a\left(H, L\to L^\prime\right) = \bigoplus_{\substack{
        \gamma \in \operatorname{Crit}\left(\mathcal{A}_H\right)\\
        \mathcal{A}_H \left(\gamma \right) < a}} \mathbb{Z}/2 \cdot \gamma,
\]
be the complex generated by only those generators of action strictly less than $a$. Since the differential is action non-increasing, it restricts to a differential on each $CW^a\left(H, L\to L^\prime\right)$, making it a subcomplex. We denote by $HW^a\left(H, L\to L^\prime\right)$ the homology of $CW^a\left(H, L\to L^\prime\right)$.
Since continuation maps do not increase action, they give rise to maps 
\[
    \iota^{H_-,H_+} \colon HW^a\left(H_-, L\to L^\prime\right) \to HW^a\left(H_+, L\to L^\prime\right),
\]
for any pair of Hamiltonians such that $H_- < H_+$.
One then defines
\[
HW^a(M,L\to L^\prime) := \varinjlim_{H} HW^a(H,L\to L^\prime),
\]
with the limit going over all $H\in \mathcal{H}_{\operatorname{reg}}(M,L\to L^\prime)$.

For every $a\le b$ the inclusion of subcomplexes \[CW^a\left(H, L\to L^\prime\right) \subseteq CW^b\left(H, L\to L^\prime\right),\] induces a map on homology, which also commutes with the corresponding continuation maps to induce a map on the direct limit of homologies. These are the persistence morphisms
\[\iota_{a\to b}\colon HW^a(M,L\to L^\prime) \to HW^b(M,L\to L^\prime).\]

One can relate the action filtration on $HW$ to the slope at infinity of the Hamiltonians participating in the limit. The discussion follows \cite[Section 2.2.4]{alves2019dynamically}, adapted to the setting where $f\equiv A$,  $g\equiv B$, constants, near $\partial M$. We denote:   $\gamma:= A-B$.
\begin{proposition}\label{prop:actionAndSlope}
For all $a\ge0$, 
    \begin{equation}\label{eqn:actionAndSlope}
    \varinjlim_{
     \substack{
         H \\
         H\text{'s slope } < a
        }}\!\!\!\!\!\!\!HW(H,L\to L^\prime) \cong
     HW^{a+\gamma}(M,L\to L^\prime).
\end{equation}
\end{proposition}
\begin{proof}
    Fix $a \in \mathbb{R}$. By the discreteness of the spectrum, there exists $\mu < a$, such that there is no Reeb chord of $\alpha^{}_{M}$ going from $\Lambda$ to $\Lambda^\prime$ with length in the interval $[\mu, a)$. We now choose an admissible Hamiltonian $H^\mu$ with slope $\mu$ at infinity such that:
    \begin{itemize}
        \item $H^\mu$ is a negative constant $-k$ in $M_{1-\eps}$ for some $\eps>0$,
        \item and in $M_{[1-\eps,1]}$, the Hamiltonian $H^\mu$ is a radial convex function, increasing sharply near $\partial M$.
    \end{itemize}  

    If $-k$ is small enough, and $H^\mu$ increases sharply enough close to $\partial M$ then we have that the action of all chords in $\mathcal{T}_{L\to L^\prime}(H)$ is less than $<a+\gamma$.

    This follows from the formula for the action, interperting the term $ \intop_0^1 \gamma^*\lambda - \intop_0^1 H\left(\gamma\left(t\right)\right)dt$ as the $y$-intercept of the tangent to the radial Hamiltonian, see \cite[Section 4.3]{ritter2013topological}. The term involving the difference of the values of $f$ and $g$ equals $\gamma$. 

    If $(M,L\to L^\prime)$ is regular, then $H^\mu$ is regular for all $\mu$, and in that case, increasing $\mu$ to a $\mu^\prime > a$, may only add Hamiltonian chords with action strictly larger than $a+\gamma$. These chords appear at a larger radial coordinate than these with action less than $a_\gamma$.

    The non-regular case is resolved by a small $C^\infty$ perturbation of $H^\mu$, after which, the above conclusion still holds.

    Thus, we get
    \begin{multline*}
        HW^{a+\gamma}(M,L\to L^\prime) \cong \varinjlim_{
     \substack{
         H
        }}HW^{a+\gamma}(H,L\to L^\prime) = \\
    =\varinjlim_{
     \substack{
         H \\
         H\text{'s slope } < a
        }}\!\!\!\!\!\!\!HW(H,L\to L^\prime).
    \end{multline*}

    The last equality is due to the no-escape lemma \cite[Section 19.5]{ritter2013topological} together with the location of the orbits of higher action guaranteeing that the action truncated complexes coincide with these of $H^\mu$ for $\mu$ close enough to $a$. See also \cite[Lemma 1.6]{viterbo2018functors}.
\end{proof}

Thus one may alternatively think of the filtration on $HW(M,L\to L^\prime)$ as given by a filtration on the slopes of the Hamiltonians taken in the limit, rather than the filtration by subcomplexes of action sublevels. This viewpoint will prove useful when we define continuation maps with respect to isotopies of exact Lagrangians.

Moreover, one can show by careful choice of cofinal sequences of Hamiltonians, together with the no-escape lemma \cite[Section 19.5]{ritter2013topological}\footnote{The idea is to select the Hamiltonians in such a way that they coincide on the part of slope $\le a$, and thus the continuation maps sends each generator to a generator represented by the same underlying orbits.}, that the isomorphism above interchanges between the persistence morphisms and the maps induced by continuations, namely that the following diagram commutes,
\begin{equation*}
    \begin{tikzcd}
        [execute at end picture={
        \draw[->] (\tikzcdmatrixname-1-3.south) ++(0,0.6) -- node[right] {$\scriptstyle\iota^{a,b}$} (\tikzcdmatrixname-2-3);
        }]
        {HW^a(M, L \to L^\prime)} \arrow[rr, "\sim"] \arrow[d, "\iota_{a \to b}"] &  & {\varinjlim\limits_{  \substack{      H \\      H\text{'s slope } < a     }}\!\!\!\!\!\!\!HW(H,L\to L^\prime)}  \\
        {HW^b(M, L \to L^\prime)} \arrow[rr, "\sim"]                              &  & {\varinjlim\limits_{  \substack{      H \\      H\text{'s slope } < b     }}\!\!\!\!\!\!\!HW(H,L\to L^\prime)}                           
    \end{tikzcd},
\end{equation*}
where the vertical arrow on the right, $\iota^{a,b}$, stands for the morphism induced by continuation morphisms given by two cofinal sequences of Hamiltonians, chosen so that the continuation is well defined.

Wrapped Floer homomolgy, is, therefore, a persistence module of finite type. That is because for every Hamiltonian $H^a$ with slope $a$ at infinity, the complex has a finite number of generators, and the "jumps" occur only at slopes belonging to the spectrum of the contact boundary, which is discrete.

\subsection{Viterbo's Transfer Map}\label{ssec:viterboTransfer}
Our goal is to describe a version of Viterbo's transfer map in wrapped Floer homology, that respects the action filtration, as in \cite{alves2019dynamically}. In our particular setting, we map from $M$ to $M_\eps$, and the Lagrangians are conical within the trivial cobordism $M_{[\eps,1]}$. Such setting allows for the transfer map to be realized solely by continuation maps, a fact which will play a role in commutativity of certain diagrams. We now outline its construction.

Unlike \cite{alves2019dynamically}, we only consider the transfer map with respect to a trivial cobordism. Additionally, we do not require one of the Lagrangians to be conical, and here the primitive functions on the Lagrangians are not necessarily zero, but rather constant in a neighborhood of the boundary.

Let us now turn to the theorem's statement and proof. \footnote{Note that our convention for subscripts and for using the prime symbol as in $L^\prime$ is different from the convention used by \cite{alves2019dynamically}.}, and prove it.
Let $M:=(Y_M,\omega_M,\lambda_M)$ be a Liouville domain and let $L^{}_M$ and $L^\prime_M$ be two disjoint asymptotically conical Lagrangians in $M$. Moreover assume that $L_\eps := L_M \cap M_\eps$ and $L^\prime_{M_\eps} := L^\prime_{M_\eps} \cap M_\eps$ are asymptotically conical in $M_\eps$. 

\begin{proposition}[Viterbo's transfer map, variation of {\cite[Section 3]{alves2019dynamically}}]   
    Assume that the Lagrangians $L_M$ and $L^\prime_M$ above, also satisfiy the properties:
    
    \begin{align}
    \begin{split}\label{transfer_admissible1}
    &\lambda\vert^{}_{L_M \setminus L_{M_\eps}} \text{ vanishes near the boundary }\partial (L_M \setminus L_{M_\eps}) = \partial L_M \cup \partial L_{M_\eps},\\
    &\text{and one can write }\lambda\vert^{}_{L_M \setminus L_{M_\eps}} = df, \text{ where }f \text{ vanishes near } \partial L_{M_\eps} \text{,}\\
    &\text{and } f\equiv A , \text{ a constant, near } \partial L_M. 
    \end{split}
    \end{align}
    
    \begin{align}
    \begin{split}\label{transfer_admissible2}
    &\lambda\vert^{}_{L^\prime_M \setminus L^\prime_{M_\eps}} \text{ vanishes near the boundary }\partial (L^\prime_M \setminus L^\prime_{M_\eps}) = \partial L^\prime_{M} \cup \partial L^\prime_{M_\eps},\\
    &\text{and one can write }\lambda\vert^{}_{L^\prime_M \setminus L^\prime_{M_\eps}} = dg, \text{ where }g \text{ vanishes near } \partial L^\prime_{M_\eps} \text{,}\\
    &\text{and } f\equiv B, \text{ a constant, near } \partial L^\prime_M. 
    \end{split}
    \end{align}
    
    \begin{align}
        \begin{split}\label{transfer_admissible3}
        &\text{The constants } A \text{ and } B \text{ satisfy } B-A\ge0.
    \end{split}
    \end{align}
    
    Then there exist morphisms
    \[{\mathcal{R}es}^{}_{M\to {M_\eps}}\colon HW\left(M,L_M \to L^\prime_M\right) \to HW\left(M_\eps,L_{M_\eps} \to L^\prime_{M_\eps}\right),\]
    compatible with the persistence morphisms, namely, for all $\mu < \nu$, the following diagram commutes:
    \begin{equation*}
        \begin{tikzcd}
            {HW^\mu\left(M,L_M \to L^\prime_M\right) } \arrow[d, "{\mathcal{R}es}^{}_{M\to M_\varepsilon}"] \arrow[rr, "\iota^{}_{\mu\to\nu}"] &  & {HW^\nu\left(M,L_M \to L^\prime_M\right) } \arrow[d, "{\mathcal{R}es}^{}_{M\to M_\varepsilon}"] \\
            {HW^\mu\left(M_\varepsilon,L_{M_\varepsilon} \to L^\prime_{M_\varepsilon}\right) } \arrow[rr, "\iota^{}_{\mu\to\nu}"]              &  & {HW^\nu\left(M_\varepsilon,L_{M_\varepsilon} \to L^\prime_{M_\eps}\right) }                    
        \end{tikzcd}
    \end{equation*}
\end{proposition}
\begin{proof}
We denote the flow of Liouville vector field of $M$ by $\phi$. Let $t$ be the Liouville coordinate on $\widehat{M}$ induced by taking the initial conditions to be at $\partial M$, namely $t$ is defined by $\phi^t(x) \colon (-\infty,\infty) \times \partial M \to  \widehat{M}$, and $\phi^0(\partial M) = \partial M$. Also let $r = e^t$ be the logarithmic Liouville coordinate,  expressed via 
\[\phi^{\log{r}}(x) \colon (0,\infty) \times \partial M \to  \widehat{M}.\]

Denote by $t^\prime$ the Liouville coordinate induced on $\widehat{M_\eps}=\widehat{M}$ by the Liouville flow with initial condition on $\partial M_\eps$, which we dentoe $\phi^\prime$, namely we have ${\phi^\prime}^{t^\prime}(x) \colon (-\infty,\infty) \times \partial M_\eps \to  \widehat{M}$ with ${\phi^\prime}^0(\partial M_\eps) =  \partial M_\eps$. We now let $r^\prime = e^{t^\prime}$ denote the logarithmic Liouville coordinate,  expressed via 
\[ {\phi^\prime}^{\log{r^\prime}}(x) \colon (0,\infty) \times \partial M_\eps \to  \widehat{M}.\]

To relate $r$ and $r^\prime$ note that by definition have that $\partial M_\eps = {\phi^{\prime}}^{\eps} (\partial M) = \phi^{\log \eps} (\partial M)$, therefore ${\phi^\prime}^{t^\prime}(\partial M_\eps) = {\phi^\prime}^{t^\prime}\phi^{\log \eps} (\partial M) = {\phi}^{t^\prime +\log \eps} (\partial M)$, so $t^\prime + \log \eps = t$. We obtain
\[
    e^{t^\prime + \log \eps} = e^t,
\]
hence,
\[
    \eps r^\prime = r,
\]
therefore $r^\prime = \frac{1}{\eps} r$.  (Note that $\eps < 1$.)

Thus, a radial Hamiltonian, $g(r^\prime) = \mu r^\prime + b$ is expressed in terms of $r$ as $h(r) = \frac{\mu}{\eps} r + b$.

For every noncharacteristic $\mu$, consider a Hamiltonian $H^\mu$, admissible for $M$ with slope $\mu$ at infinity, that is, $H^\mu$ such that:
    \begin{enumerate}
        \item $H^\mu$ equals a negative constant $-k$ in $M_{1-\delta}$ for some $\delta>0$.
        \item In $M_{[1-\delta,1]}$, the Hamiltonian $H^\mu$ is a radial convex function, increasing sharply near $\partial M$.
        \item $H^\mu(r,x)= h(r)=\mu r + b$ on $[1,\infty) \times \partial M$, with $b = -k -\mu$, where $x$ denotes the coordinate on the contact manifold $\partial M$.
    \end{enumerate} 

Also choose a Hamiltonian $G^\mu$ admissible for $M_\eps$ with slope $\mu$ at infinity, that is, such that:
    \begin{enumerate}
        \item $G^\mu$ is a negative constant $-k^\prime$ in $M_{\eps-\delta}$ for some $\delta>0$,
        \item In $M_{[\eps-\delta,\eps]}$, the Hamiltonian $G^\mu$ is a radial convex function, increasing sharply near $\partial M_\eps$.
         \item $G^\mu(r^\prime,x^\prime)= g(r^\prime)=\mu r^\prime + b^\prime$ on $[1,\infty) \times \partial M_\eps$, with $b^\prime = -k^\prime -\mu$, where $x^\prime$ denotes the coordinate on the contact manifold $\partial M_\eps$.
    \end{enumerate} 

Note that when expressed by the $r$ coordinate, $G^\mu$ becomes 
\[G^\mu(r,x^\prime):=G^\mu(r^\prime(r),x^\prime)= g(r^\prime(r))=\mu r^\prime(r) + b^\prime = \frac{\mu}{\eps} r + b^\prime, \]
on $[\eps,\infty) \times \partial M_\eps$. By choosing $-k^\prime > -k$, we obtain $G^\mu > H^\mu$ on $\widehat{M}$, since $\frac{1}{\eps}>1$.

Therefore, we may pick two monotone sequences of  Hamiltonians, $H^\mu$ and $G^\mu$, such that $G^\mu > H^\mu$ on $\widehat{M}$ and such that for every $a \in \mathbb{R}$
\[
\varinjlim_{\mu \to \infty} HW^a(H^\mu,L\to L^\prime) = HW^a(M,L\to L^\prime),
\]
and 
\[
\varinjlim_{\mu \to \infty} HW^a(G^\mu,L\to L^\prime) = HW^a(M_\eps,L\cap M_\eps \to L^\prime \cap M_\eps).
\]
For the latter, note that the "no escape lemma", Lemma \ref{lem:lagNoEscape}, guarantees that the chain complex for $CW^a(G^\mu, L\cap M_\eps \to L^\prime \cap M_\eps)$ does not depend on the extension of the Lagrangians to $\widehat{M_\eps}$, as long as the extension is exact, since they are conical near $\partial M_\eps$.

Picking monotone homotopies from $H^\mu$ to $G^\mu$, we obtain commuting squares of continuation maps for every $\mu < \nu$:
\begin{equation*}
    \begin{tikzcd}
        {HW^a(H^\mu, L \to L^\prime)} \arrow[rr, "\iota^{H^\mu \to H^\nu}"] \arrow[d, "\iota^{H^\mu \to G^\mu}"] &  & {HW^a(H^\nu, L \to L^\prime)} \arrow[d, "\iota^{H^\nu \to G^\nu}"] \\
        {HW^a(G^\mu, L \to L^\prime)} \arrow[rr, "\iota^{G^\mu \to G^\nu}"]                                      &  & {HW^a(G^\nu, L \to L^\prime)}                                     
    \end{tikzcd}
\end{equation*}

Note that the energy filtration is indeed respected by the maps. This follows from the action-energy computation for a continuation solution with a monotone homotopy, together with the fact that $B-A\ge 0$ ensures an action inequality.

These squares induce maps between the direct limits, namely, the Viterbo transfer maps:
\[
{\mathcal{R}es}^{}_{M\to M_\eps}\colon HW^a\left(M,L \to L^\prime\right) \to HW^a\left(M_\eps,L\cap M_\eps \to L^\prime\cap M_\eps\right),
\]
which, as we have shown, preserve the action filtration.
\end{proof}

\begin{remark}
    Unlike the general construction of Viterbo's transfer map, as in \cite{alves2019dynamically}, which uses both a continuation and a projection on a quotient complex, (or, in some versions, restriction to subcomplex), In the setting of a trivial cobordism, as we have shown, the transer map can be defined purely by a continuation map. This feature is crucial to our proof of commutativity of certain diagrams involving ${\mathcal{R}es}$, in Propositions \ref{prop:contAndContinuationCommute} and \ref{prop:conicalCommDiagram}.
\end{remark}

We also recall the following facts from \cite{alves2019dynamically} about the transfer map, which will be crucial in the proof of Theorem \ref{thm:mainTheorem}:
\begin{lemma}[\cite{alves2019dynamically} Lemma 3.1]\label{lem:viterboRestriction}
    Assume that $L, L^\prime$ are conical in $M_{[1-\eps,1]} = M\setminus M_\eps$, for $0<\eps <1$, then:
    \begin{enumerate}
        \item For all $\mu>0$ there exists an isomorphism of cohomology groups, denoted by $\varphi_\mu$:
            \[
            \operatorname{HW}^\mu\left(M_\eps,L\cap M_\eps \to L^\prime \cap M_\eps\right) \overset{\varphi_\mu}{\cong} \operatorname{HW}^{\frac{1}{\eps}\mu}\left(M,L\to L^\prime\right). 
            \]
        These isomorphisms are compatible with the persistence morphisms in the sense that for every $\mu < \nu$ the following diagram commutes:
        \begin{equation*}
            \begin{tikzcd}
            {HW^{\mu}\left(M_\varepsilon,L\cap M_\varepsilon\to L^\prime\cap M_\varepsilon\right)} \arrow[d, "\rotatebox{90}{$\sim$}", "\varphi_\mu"'] \arrow[rr, "\iota^{}_{\mu\to\nu}"] &  & {HW^{\nu}\left(M_\varepsilon,L\cap M_\varepsilon\to L^\prime\cap M_\varepsilon\right)} \arrow[d, "\rotatebox{90}{$\sim$}", "\varphi_\mu"'] \\
            {HW^{\frac{1}{\varepsilon}\mu}\left(M,L\to L^\prime\right)} \arrow[rr, "\iota^{}_{\frac{1}{\varepsilon}\mu\to\frac{1}{\varepsilon}\nu}"]                      &  & {HW^{\frac{1}{\varepsilon}\nu}\left(M,L\to L^\prime\right).}                                                               
            \end{tikzcd}                                                        
        \end{equation*}
        
        \item Moreover, the composition of the Viterbo restriction map 
        \[\operatorname{HW}^{\mu} \left(M,L\to L^\prime\right) \xrightarrow{{\mathcal{R}es}^{}_{M\to M\varepsilon}} \operatorname{HW}^\mu\left(M_\eps,L\cap M_\eps \to L^\prime \cap M_\eps\right),\]
        with the above map, $\varphi_\mu$, equals the persistence morphism
        \pushQED{\qed}
        \[
            \operatorname{HW}^{\mu} \left(M,L\to L^\prime\right) \xrightarrow{\operatorname{\iota}_{\mu\to \frac{1}{\eps}\mu}} \operatorname{HW}^{\frac{1}{\eps}\mu} \left(M,L\to L^\prime\right). \qedhere
        \]
        \popQED
    \end{enumerate}
\end{lemma}
\begin{remark}
    The second statement of Item 1 in the above Lemma is not explicitly stated in \cite{alves2019dynamically} Lemma 3.1, but it follows from the proof of the first item, since the rescaling map provides a bijection both for Floer and continuation solutions, between the solutions before and after the rescaling.
\end{remark}

\begin{remark}
    Note that while our map is seemingly defined differently from that of \cite{alves2019dynamically}, in the case of a trivial cobordism one can realize the projection map onto a subcomplex, via a composition with a continuation map, hence in homology the two maps coincide.
\end{remark}

\subsection{Continuation Maps Associated to an Exact Lagrangian Isotopy}\label{ssec:continuationMovingBdry}
In this subsection we introduce morphisms in wrapped Floer homology, which are induced by exact Lagrangian isotopies. They are constructed via the count of Floer solutions with moving boundary conditions. See \cite{oh1993floer} for the original construction in Floer theory in the closed monotone setting and \cite{oh2009unwrapped} for a far reaching generalization to open manifolds. Since the author knows no reference that deals with our specific generality, we give the construction in full detail, based on the construction in \cite{oh1993floer}.

Consider an exact asymptotically conical Lagrangian $L$, and denote by $f$ a primitive $\lambda\vert^{}_{L}=df$ where $f\colon L \to \mathbb{R}$. Let $L^\prime_0$ be another exact asymptotically conical Lagrangian, and let $\left\lbrace L^\prime_\tau \right\rbrace_{0\le \tau\le T}$ be an isotopy of exact asymptotically conical Lagrangians, with $\lambda\vert^{}_{L_\tau} = dg_\tau$ for $g_\tau \colon L^\prime_\tau \to \mathbb{R}$.
Assume that the isotopy $\left\lbrace L^\prime_\tau \right\rbrace_{0\le \tau\le T}$ is supported away from the boundary of $M$, i.e. it is constant in a neighborhood of $\partial M$.

Note that in particular it follows that $f$ and $g_\tau$ are constant near $\partial M$. Let us denote by $A$ the value of $f$ near $\partial M$ and by $B_\tau$ the value of $g_\tau$ near $\partial M$.

As the domain of Floer solution is unbounded, namely $\mathbb{R} \times [0,1]$, let us reparameterize the isotopy accordingly. Choose a smooth function $\beta(\tau) \colon \mathbb{R} \to \mathbb{R}$ such that $\beta \equiv 0$ for $\tau\in (-\infty,0]$, $\beta \equiv T$ for $\tau \in [1,\infty)$, and in $[0,1]$, $\beta(\tau)$ is monotone increasing, mapping the interval $[0,1]$ onto $[0,T]$. We now  consider the isotopy $\lbrace L^\prime_{\beta(\tau)} \rbrace_{\tau \in \mathbb{R}}$.

By \cite[Exercise 11.3.17]{mcduff2017introduction} there exists a Hamiltonian isotopy $\Psi_\tau\colon M \to M$, such that $L^\prime_{\beta(\tau)} = \Psi_\tau \left( L^\prime_0 \right)$. Denote by $G_\tau \colon M\times [0,1] \to \mathbb{R}$ the Hamiltonian generating the isotopy $\Psi$. We can assume that $G_\tau$ is supported away from the boundary of $M$, by multiplying it by a suitable cutoff function.

We extend the exact Lagrangians and the isotopy to the completion, $\widehat{M}$, by completing the Lagrangians, $\widehat{L}$, $\widehat{L^\prime_0}$, and all of $\widehat{L^\prime_\tau}$, and extending the isotopy by identity. Such extension is posslible since the isotopy was chosen to be identity near $\partial M$.

We consider the moduli space of Floer solutions with moving boundary conditions, of Fredholm index 0, connecting $\gamma_-$ and $\gamma_+$. We denote it by $\mathcal{M}^0\left(\gamma_-,\gamma_+, H, J, \lbrace \widehat{L}^\prime_{\beta(\tau)} \rbrace_{\tau\in\mathbb{R}} \right)$, and its elements are the Fredholm index $0$ solutions satisfying:
\begin{equation}\label{eqn:floerMovingBoundary}
    \begin{aligned}
        &u\colon \mathbb{R}\times [0,1] \to \widehat{M}, \\
        &\overline{\partial}^{}_{J,H}u := \partial_s u + J\left( \partial_t u - X_H\left(u\right) \right) = 0, \\
        &\lim_{s\to\pm \infty} u(s,t) = \gamma_{\pm}(t), \\
        &u(s,0)\in \widehat{L} \text{ and }  u(s,1)\in \widehat{L^\prime}^{}_{\!\!\beta(s)}
    \end{aligned}  
\end{equation}

We define a morphism $\mathcal{C}ont_{\{L^\prime_\tau\}_{0}^{T}}\colon CF(H, \widehat{L}\to \widehat{L^\prime_0}) \to  CF(H, \widehat{L}\to \widehat{L^\prime_T})$ by counting the elements of $\mathcal{M}^0\left(\gamma_-,\gamma_+, H, J, \lbrace \widehat{L^\prime}^{}_{\!\!\beta(\tau)} \rbrace_{\tau\in\mathbb{R}} \right)$ :
\[
    \mathcal{C}ont_{\{L^\prime_\tau\}_{0}^{T}}\, x \quad = \sum_{x^\prime \in \mathcal{P}^H_{\widehat{L}\to \widehat{L^\prime_T}}} \#_2\mathcal{M}^0\left(x,x^\prime, H, J, \lbrace \widehat{L^\prime}^{}_{\!\!\beta(\tau)} \rbrace_{\tau\in\mathbb{R}} \right) \cdot x^\prime.
\]

All of the analysis of such solutions is the same as in the case of fixed boundary conditions, namely, bubbling is local, and gluing and compactness analysis is unchanged, since breaking and gluing happen at the ends, where the boundary conditions are fixed and do not depend on $s$. See \cite{oh1993floer}. 

The deviations from the standard treatment of continuation solutions are twofold. First, in establishing a priori $C^0$-bounds, ensuring that solutions do not escape to infinity, and second, in establishing a priori energy bounds to ensure the applicability of Gromov compactness. These bounds are addressed in \ref{ssec:energyAndC0}.
The same principles hold for the other moduli spaces of moving boundary solutions, introduced later in the text.

Using standard gluing and breaking arguments, we prove that:

\begin{proposition}\label{prop:contAndContinuationCommute}
    $\mathcal{C}ont_{\{L^\prime_\tau\}_{0}^{T}}$ is a chain map. Moreover, in homology it commutes with the continuation maps induced by a nondecreasing homotopy from $H$ to $H^\prime$, for $H \le H^\prime$, namely the following diagram commutes:
     \begin{equation*}
        \pushQED{\qed}
        \begin{tikzcd}
            {HW(H, \widehat{L}\to \widehat{L^\prime_0})} \arrow[rr, "\mathcal{C}ont_{\{L^\prime_\tau\}_{0}^{T}}"] \arrow[d, "{\iota^{H,H^\prime}}"] &  & {HW(H, \widehat{L}\to \widehat{L^\prime_T})} \arrow[d, "{\iota^{H,H^\prime}}"] \\
            {HW(H^\prime, \widehat{L}\to \widehat{L^\prime_0})} \arrow[rr, "\mathcal{C}ont_{\{L^\prime_\tau\}_{0}^{T}}"]                            &  & {HW(H^\prime, \widehat{L}\to \widehat{L^\prime_T})}.&\qedhere 
            \end{tikzcd} 
            \popQED
     \end{equation*}
\end{proposition}
\begin{proof}
    Recall that due to the Lagrangians being exact, and $\omega$ being exact, there is no disc bubbling nor sphere bubbling. This means that all relevant moduli spaces can be compactified using only broken solutions. Moreover, breaking and gluing occurs at the ends, at which we impose the usual Lagranian boundary conditions, not the moving ones, and hence the analysis of breaking and gluing is the same as in the case of standard continuation maps.

    The proof that $\mathcal{C}ont_{\{L^\prime_\tau\}_{0}^{T}}$ is a chain map, has structure identical to the proof that any continuation map is a chain map. One considers the compactified moduli space solutions of equations (\ref{eqn:floerMovingBoundary}) of Fredholm index $1$. The zero-dimensional boundary is composed of pairs $(u,v)$ where either:
    \begin{enumerate}
        \item $u$ is a Floer solution with asymptotes $\gamma_-$ and $\gamma_0$, and $v$ is a continuation solution with moving boundary (Equation (\ref{eqn:floerMovingBoundary})), with asymptotes $\gamma_0$ and $\gamma_+$, for $\gamma_-,\gamma_0,\gamma_+$ some Hamiltonian chords.
        \item $u$ is a continuation solution with moving boundary (Equation (\ref{eqn:floerMovingBoundary})), with asymptotes $\gamma_-$ and $\gamma_0$, and $v$ is a Floer solution with asymptotes $\gamma_0$ and $\gamma_+$, for $\gamma_-,\gamma_0,\gamma_+$ some Hamiltonian chords.        
    \end{enumerate}
    The count of this zero-dimensional boundary gives the desired relations.

    To prove the commutativity of 
    \begin{equation*}
        \begin{tikzcd}
            {HW(H, \widehat{L}\to \widehat{L^\prime_0})} \arrow[rr, "\mathcal{C}ont_{\{L^\prime_\tau\}_{0}^{T}}"] \arrow[d, "{\iota^{H,H^\prime}}"] &  & {HW(H, \widehat{L}\to \widehat{L^\prime_T})} \arrow[d, "{\iota^{H,H^\prime}}"] \\
            {HW(H^\prime, \widehat{L}\to \widehat{L^\prime_0})} \arrow[rr, "\mathcal{C}ont_{\{L^\prime_\tau\}_{0}^{T}}"]                            &  & {HW(H^\prime, \widehat{L}\to \widehat{L^\prime_T})},&\
            \end{tikzcd} 
     \end{equation*}
 we consider the maps on the chain level:
    \begin{equation*}
        \begin{tikzcd}
            {CW(H, \widehat{L}\to \widehat{L^\prime_0})} \arrow[rr, "\mathcal{C}ont_{\{L^\prime_\tau\}_{0}^{T}}"] \arrow[d, "{\iota^{H,H^\prime}}"] &  & {CW(H, \widehat{L}\to \widehat{L^\prime_T})} \arrow[d, "{\iota^{H,H^\prime}}"] \\
            {CW(H^\prime, \widehat{L}\to \widehat{L^\prime_0})} \arrow[rr, "\mathcal{C}ont_{\{L^\prime_\tau\}_{0}^{T}}"]                            &  & {CW(H^\prime, \widehat{L}\to \widehat{L^\prime_T})}.&
        \end{tikzcd} 
     \end{equation*}

    By gluing, we show that each the two compositions, namely, the one via the top path and the one via the bottom path, equal each, respectively, to one of two maps depicted below:
     \begin{equation*}
        \begin{tikzcd}
            {CF(H, \widehat{L}\to \widehat{L^\prime_0})} \arrow[rr, "\Psi^{top}", bend left] \arrow[rr, "\Psi^{bot}", bend right] &  & {CF(H^\prime, \widehat{L}\to \widehat{L^\prime_T})}
        \end{tikzcd}
    \end{equation*}
    where the maps $\Psi^{top}$ and $\Psi^{bot}$ are defined by counts of index $0$ solutions to certain continutation maps with moving boundary conditions, described as follows:

As in \ref{ssec:continuationMovingBdry} choose a smooth function $\beta(\tau) \colon \mathbb{R} \to \mathbb{R}$ such that $\beta \equiv 0$ for $\tau\in (-\infty,0]$, $\beta \equiv T$ for $\tau \in [1,\infty)$, and in $[0,1]$, $\beta(\tau)$ is monotone increasing, mapping the interval $[0,1]$ onto $[0,T]$. The isotopy $\lbrace L^\prime_{\beta(\tau)} \rbrace_{\tau \in \mathbb{R}}$, is given by $L^\prime_{\beta(\tau)} = \phi^\tau_{G_\tau} \left( L^\prime_0 \right)$. For a Hamiltonian $G_\tau \colon M\times [0,1] \to \mathbb{R}$, which is assumed to be supported away from the boundary of $M$.

Recall that in defining $\iota^{H,H^\prime}$ we choose a generic increasing homotopy through admissible Hamiltonians, $\left(H_s\right)_{s\in\mathbb{R}}$, such that $\partial_s H_s \ge 0$ with $H_s = H$ near $-\infty$, and $H_s = H^\prime$ near $\infty$, and count the associated continuation solutions.

By gluing, (see e.g. \cite[Proposition 19.5.1]{oh2015symplectic}, for an analogous result in the context of continuation maps in Hamiltonian Floer homology),  we have that for some $R>0$ large enough, which also can be taken large enough so that $H_{s-R} = H$ for $s\le 1$, the map $\Psi^{top}$ is given by 
\[
    \Psi^{top}\, x \quad = \sum_{x^\prime \in \mathcal{P}^H_{\widehat{L}\to \widehat{L^\prime_T}}} \#_2\mathcal{M}^0\left(x,x^\prime, H_{s-R}, J, \lbrace \widehat{L^\prime}^{}_{\!\!\beta(\tau)} \rbrace_{\tau\in\mathbb{R}} \right) \cdot x^\prime,
\]

where, for a given homotopy of Hamiltonians $H_s$, the space  
\begin{equation}\label{eqn:continuationMovingBoundaryNotation}\mathcal{M}^0\left(\gamma_-,\gamma_+, H_{s}, J, \lbrace \widehat{L^\prime}^{}_{\!\!\beta(\tau)} \rbrace_{\tau\in\mathbb{R}} \right),\end{equation}
is defined to be the moduli space of Floer continuation solutions with moving boundary conditions, of Fredholm index $0$, connecting $\gamma_-$ and $\gamma_+$, solving:
\begin{equation}\label{eqn:continuationMovingBoundary}
    \begin{aligned}
        &u\colon \mathbb{R}\times [0,1] \to \widehat{M}, \\
        &\overline{\partial}^{}_{J,H}u := \partial_s u + J\left( \partial_t u - X_{H_{s}}\left(u\right) \right) = 0, \\
        &\lim_{s\to\pm \infty} u(s,t) = \gamma_{\pm}(t), \\
        &u(s,0)\in \widehat{L} \text{ and }  u(s,1)\in \widehat{L^\prime}^{}_{\!\!\beta(s)}.
    \end{aligned}  
\end{equation}

Similarly, $\Psi^{bot}$ is given by 
\[
    \Psi^{bot}\, x \quad = \sum_{x^\prime \in {\scalebox{0.9}{$ \mathcal{P}$}}^H_{\widehat{L}\to \widehat{L^\prime_T}}} \#_2\mathcal{M}^0\left(x,x^\prime, H_{s+R}, J, \lbrace \widehat{L^\prime}^{}_{\!\!\beta(\tau)} \rbrace_{\tau\in\mathbb{R}} \right) \cdot x^\prime,
\]
for $R$ large enough, which can also be assumed to be large enough so that $H_{s+R}= H^\prime$ for $s\ge 0$.

To finish the proof of commutativity, we present a chain homotopy 
\[\mathfrak{H}:{CF(H, \widehat{L}\to \widehat{L^\prime_0})} \to {CF(H^\prime, \widehat{L}\to \widehat{L^\prime_T})},\]
between $\Psi_{top}$ and $\Psi_{bot}$, namely, $d\mathfrak{H} + \mathfrak{H}d = \Psi_{top} - \Psi_{bot}$.

To define $\mathfrak{H}$ we construct an appropriate moduli space; the map is defined by counting its points. Consider a generic homotopy of monotone homotopies, through addmissible hamiltonians, $H_{s,\theta}$, with $\theta \in [0,1]$, such that $H_{s,0} = H_{s-R}$ and $H_{s,1} = H_{s+R}$.

We define the space  $\mathcal{M}^{-1}\left(\gamma_-,\gamma_+, H_{s,\theta}, J, \lbrace \widehat{L^\prime}^{}_{\!\!\beta(\tau)} \rbrace_{\tau\in\mathbb{R}} \right)$ to be the moduli space of pairs $(u,\theta)$ where $u$ is a Floer continuation solution with moving boundary conditions, of Fredholm index $-1$ (i.e. as a pair $(u,\theta)$ it has index $0$), connecting $\gamma_-$ and $\gamma_+$, solving:
\begin{equation}\label{eqn:chainhomotopyMovingBoundary}
    \begin{aligned}
        &u\colon \mathbb{R}\times [0,1] \to \widehat{M}, \\
        &\overline{\partial}^{}_{J,H_{s,\theta}}u := \partial_s u + J\left( \partial_t u - X_{H_{s,\theta}}\left(u\right) \right) = 0, \\
        &\lim_{s\to\pm \infty} u(s,t) = \gamma_{\pm}(t), \\
        &u(s,0)\in \widehat{L} \text{ and }  u(s,1)\in \widehat{L^\prime}^{}_{\!\!\beta(s)},
    \end{aligned}  
\end{equation}
and set 
\[
    \mathfrak{H}\, x \quad = \sum_{x^\prime \in \mathcal{P}^H_{\widehat{L}\to \widehat{L^\prime_T}}} \#_2\mathcal{M}^{-1}\left(\gamma_-,\gamma_+, H_{s}, J, \lbrace \widehat{L^\prime}^{}_{\!\!\beta(\tau)} \rbrace_{\tau\in\mathbb{R}} \right) \cdot x^\prime.
\]

Now consider  $\overline{\mathcal{M}^{0}\left(\gamma_-,\gamma_+, H_{s,\theta}, J, \lbrace \widehat{L^\prime}^{}_{\!\!\beta(\tau)} \rbrace_{\tau\in\mathbb{R}} \right)}$, which is the closure of the moduli space of index $1$ pairs $(u,\theta)$, defined analogously to the moduli space above. I.e. $u$ is an index $0$ solution to (\ref{eqn:chainhomotopyMovingBoundary}) with a given $\theta$. Let us now describe the points in the zero-dimensional boundary. They are of one of several types:

  \begin{enumerate}
        \item Pairs $(v,w)$ where either:
        \begin{enumerate}
            \item $v$ is a Floer solution with asymptotes $\gamma_-$ and $\gamma_0$, and $w=(u,\theta)$ is a solution to Equation (\ref{eqn:chainhomotopyMovingBoundary}) (chain homotopy with with moving boundary equation), with asymptotes $\gamma_0$ and $\gamma_+$, for $\gamma_-,\gamma_0,\gamma_+$ some Hamiltonian chords.
            \item $v=(u,\theta)$ is a solution to Equation (\ref{eqn:chainhomotopyMovingBoundary}) (chain homotopy with with moving boundary equation), with asymptotes $\gamma_-$ and $\gamma_0$, and $v$ is a Floer solution with asymptotes $\gamma_0$ and $\gamma_+$, for $\gamma_-,\gamma_0,\gamma_+$ some Hamiltonian chords.      
        \end{enumerate}
        \item Solutions $u$ solving Equation (\ref{eqn:continuationMovingBoundary}) with parameters corresponding to $\Psi^{top}$.
        \item Solutions $u$ solving Equation (\ref{eqn:continuationMovingBoundary}) with parameters corresponding to $\Psi^{bot}$
    \end{enumerate}
 The boundary relation thus implies the algebraic relation $d\mathfrak{H} + \mathfrak{H}d = \Psi_{top} - \Psi_{bot}$.

\end{proof}

Picking $H$ to be an admissible Hamiltonian with slope $a$ at infinity and $H^\prime$ to be an admissible Hamiltonian with slope $b$ at infinity, by the above proposition, and by the isomorphism (\ref{eqn:actionAndSlope}) identifying the action filtration with the slope filtration, we deduce the following proposition:

\begin{proposition}\label{prop:movingBdryContPerModMorphism} 
    Recall our notation $B_\tau$ for the constant value of $g_\tau$ near $\partial M$. We denote $\beta:= B_T - B_0$.
    The morphisms $\mathcal{C}ont_{\{L^\prime_\tau\}_{0}^{T}}$ induce a persistence module morphism from $HW(M,L\to L_0^\prime)$ to $HW(M,L\to L_T^\prime)^{[+(-\beta)]}$, namely, for all $a\le b$ the following diagram commutes:
    \begin{equation*}
        %
        \pushQED{\qed}
        \begin{tikzcd}
            {HW^a(M, L \to L^\prime_0)} \arrow[rr, "\mathcal{C}ont_{\{L^\prime_\tau\}_{0}^{T}}"] \arrow[d, "\iota_{a\to b}"] &  & {HW^{a-\beta}(M, L \to L^\prime_T)} \arrow[d, "\iota_{(a-\beta)\to (b-\beta)}"] \\
            {HW^b(M, L \to L^\prime_0)} \arrow[rr, "\mathcal{C}ont_{\{L^\prime_\tau\}_{0}^{T}}"]                             &  & {HW^{b-\beta}(M, L \to L^\prime_T)}& \qedhere                                    
        \end{tikzcd}
        \popQED
    \end{equation*}
\end{proposition}
\begin{proof}
    The proof amounts to comparing the action values corresponding to a Hamiltonian chords of a Hamiltonian with given slope at infinity, in the case of $L\to L^\prime_0$ and in case of $L\to L^\prime_T$. Let $H_\mu$ be an admissible Hamiltonian with slope $\mu$ at infinity. Then all the generators of $CF(H_\mu, \widehat{L}\to \widehat{L^\prime_0})$ have actions less then $A-B_0 + \mu$, while all generators of $CF(H_\mu, \widehat{L}\to \widehat{L^\prime_T})$ have actions less than
    \[A-B_T + \mu = A-B_0 + B_0 - B_T + \mu= A-B_0 +\mu - \beta.\]
    Here we have used that the action equals the difference of $f$ and $g$, plus the $y$-intercept of the radial Hamiltonian. By the correspondence between slope and action filtration, namely Proposition \ref{prop:actionAndSlope}, we deduce that $\mathcal{C}ont_{\{L^\prime_\tau\}_{0}^{T}}$ preserves the action filtration up to a shift by $-\beta$. 
\end{proof}

\subsection{Energy Bounds and \texorpdfstring{$C^0$}{C0} bounds.}\label{ssec:energyAndC0}
In this section we establish the two types of a priori bounds required for the compactness/breaking mechanism to hold in the case of  $\mathcal{C}ont_{\{L^\prime_\tau\}_{0}^{T}}$, for a given isotopy of exact Lagrangians as in (\ref{eqn:floerMovingBoundary}). We need two kinds of bounds: a $C^0$ bound, namely that all the Floer trajectories with moving boundary conditions are contained in a compact set, and an a priori bound for energy of such Floer trajectories. The next two subsections are dedicated to establishing these bounds.

\subsubsection{\texorpdfstring{$C^0$}{C0} Bound for Floer Solutions with Moving Boundary Conditions}
We prove the following proposition:
\begin{proposition}\label{prop:c0Cofinement}
    Given an admissible Hamiltonian $H$ and an asymptotically cylindrical almost complex structure $J$, all solutions of (\ref{eqn:floerMovingBoundary}) have images contained in the compact part $M\subset \widehat{M}$.
\end{proposition}

Before presenting the proof we recall the "no-escape lemma" from \cite{ritter2013topological}. We give a formulation specialized to the Floer strips appearing in our setting of asymptotically conical Lagrangians:

Let $(V, d\theta)$ be an exact symplectic manifold with nonempty boundary $\partial V$ of negative contact type: the Liouville vector field points strictly inwards.
Assume $J$ is cylindrical near $\partial V$.
Moreover, suppose that $H$ is a  Hamiltonian linear in the $r$-coordinate near $\partial V$.

Let $S\subset \mathbb{R}\times [0,1] \subset \mathbb{C}$ be a Riemann surface with boundary and corners, with the Riemann surface structure induced by the inclusion into $\mathbb{R}\times [0,1]$.
The complement of the corners inside the boundary is a disjoint union of circles and open intervals. Consider the collection of boundary pieces comprising the circles and the closures of the aforementioned intervals. Assume each piece is equipped with a label $b$ or $l$ such that every two pieces intersecting in a corner carry different labels. \footnote{The letters $l,b$ were chosen because we will require the pieces labeled $l$ to map to a Lagrangian and the pieces labeled $b$ to map to the boundary $\partial V$.} Denote by $\partial_b S$ the union of the pieces labeled $b$ and by $\partial_l S$ the union of the pieces labeled $l$.

Let $L$ be an exact Lagrangian (possibly disconnected), such that $\theta\vert^{}_L = df$ with $f\vert^{}_L = 0$ near $\partial L$.  
\begin{lemma}[{\cite[Lemma 19.6]{ritter2013topological}}]\label{lem:lagNoEscape}
Let $u \colon S \to V$ be a solution of the Floer equation $\partial_s u + J\left( \partial_t u - X_H\left(u\right) \right) = 0$ with $u\left(\partial_b S\right) \subset \partial V$ and $u\left(\partial_l S\right) \subset L$, then $u\left(S\right) \subset \partial V$. \qed
\end{lemma}

We now turn to the proof of the proposition.
\begin{proof}[Proof of Proposition \ref{prop:c0Cofinement}]
    We let $V$ denote the non-compact symplectic manifold $M_{[1+\eps,\infty)}$, namely the complement of the interior of $M_{1+\eps}$ in $\widehat{M}$, and let $u\colon \mathbb{R}\times[0,1] \to \widehat{M}$ be a solution to Equations (\ref{eqn:floerMovingBoundary}).
    
    Denote by $S \subset \mathbb{R} \times [0,1]$ the preimage $u^{-1}\left(V\right)$. The boundary of the symplectic manifold $V$ has negative contact type, and the almost complex structure $J$ is conical near $\partial V$. Moreover, since $H$ is admissible $H$ is linear near $\partial V$.
    
    Note that when restricted to $V$, the isotopy $\widehat{L^\prime_t}$ is constant and equals $\widehat{L^\prime_0}$. Hence $u\vert^{}_S$ solves the usual Floer equation with non-moving Lagrangian boundary conditions.
    
    Moreover, since $\widehat{L}\vert^{}_V$ and $\widehat{L^\prime_0}\vert^{}_V$ are conical, they are both exact with a primitive function $f=0$. Hence, the "no-escape Lemma" (Lemma \ref{lem:lagNoEscape} above), holds and $u(S) \subset \partial V$, therefore $\operatorname{im} u \subset M_{1+\eps}$.  Since this is true for all $\eps >0$, it follows that $\operatorname{im} u \subset M$.
\end{proof}

\subsubsection{Energy Bound}
Fix an isotopy of exact Lagrangians, $\left\lbrace L^\prime_\tau\right\rbrace^{}_{0\le\tau\le T}$, with primitives $g_\tau \colon L^\prime_\tau \to \mathbb{R}$ satisfying $dg_\tau = \lambda\vert^{}_{L^\prime_\tau}$, and reparamterize it using a choice of a monotone function $\beta(\tau)$, mapping $[0,1]$ onto $[0,T]$ to obtain $\lbrace L^\prime_{\beta(\tau)} \rbrace_{\tau \in \mathbb{R}}$, as described in subsection \ref{ssec:continuationMovingBdry}.
 
The purpose of this section is to prove the following proposition:
\begin{proposition}
    For a given an admissible Hamiltonian $H$, and an asymptotically cylindrical almost complex structure $J$,
    there exists a constant $C$ independent of $u$ such that for every solution $u$ of Equation (\ref{eqn:floerMovingBoundary}), with moving boundary conditions, the following energy bound holds:
    \[E(u) \le \mathcal{A}^{L \to L^\prime_T}_H\left(\gamma_+\right) - \mathcal{A}^{L \to L^\prime_0}_H\left(\gamma_+\right) + C.\]
    In particular there is an a priori energy bound for solutions of (\ref{eqn:floerMovingBoundary}).
\end{proposition}
\begin{proof}
    The proof is based on ideas appearing in \cite{oh1993floer}, adapted to the setting of exact Lagrangians in Liouville domains.
    
    Let $u\colon \mathbb{R}\times[0,1] \to \widehat{M}$ be a solution of (\ref{eqn:floerMovingBoundary}), namely:
    \begin{equation*}
        \begin{aligned}
            &\overline{\partial}^{}_{J,H}u := \partial_s u + J\left( \partial_t u - X_H\left(u\right) \right) = 0, \\
            &\lim_{s\to\pm \infty} u(s,t) = \gamma_{\pm}(t), \\
            &u(s,0)\in \widehat{L} \text{ and }  u(s,t)\in \widehat{L}^\prime_{\beta(s)}.
        \end{aligned}  
    \end{equation*}
    We denote by $\Gamma$ the Riemann surface with boundary: $\mathbb{R}\times [0,1]$.
    
    Recall the definition of the energy of a Floer solution,
    \[
        E\left( u \right) = \intop_{\Gamma} \left\Vert \partial_s u \right\Vert^2_J ds\,dt = \intop_{\Gamma} \omega\left(J\partial_s u, \partial_t u\right) ds\,dt.
    \]
    Since $u$ solves the Floer equation we have that
    \begin{multline*}
        \intop_{\Gamma} \omega\left(J\partial_s u, \partial_s u\right) dsdt = \intop_{\Gamma} \omega\left(\partial_t u - X_H, \partial_s u\right) dsdt = \\ =\intop_0^1\intop_{-\infty}^{\infty} \frac{\partial}{\partial s}
         H(u) ds dt - \intop_\Gamma u^*\omega,  
    \end{multline*}
    hence 
    \[
        E\left( u \right) = \intop_0^1 H\left( \gamma_+\left(t\right)\right) dt - \intop_0^1 H\left( \gamma_-\left(t\right)\right) dt - \int_\Gamma u^*\omega.
    \]

    We now turn to estimating $\int_\Gamma u^*\omega$. To this end we consider the triangle
    \[ \Delta = \left\lbrace \left(s,t\right) \in \mathbb{R}^2 \,\middle\vert\, 0\le s \le 1,\; 1\le t \le s + 1\right\rbrace,
    \]
    and we extend the map $u$ to a map $w\colon \Gamma \cup \Delta \to \widehat{M}$ defined by
    \[ w(s,t)=
    \begin{cases}
       u(s,t) & (s,t) \in \Sigma \\
       \left(\Phi^{t-1}_{G_{t-1}}\right)^{-1} \! \left(u\left(s,1\right)\right) & (s,t) \in \Delta,
    \end{cases}
    \]
    where $G_\tau$ is the Hamiltonian realizing the isotopy $\bigl\lbrace\widehat{L}^\prime_{\beta(\tau)}\bigr\rbrace_{0\le \tau \le 1}$, as 
    \[
    \widehat{L}^\prime_{\beta(\tau)} = \Phi^\tau_{G_\tau}\left(\widehat{L}^\prime_{\beta(0)}\right).
    \]
    Here $\Phi^\tau_{G_\tau}$ is the Hamiltonian flow generated by $G_\tau$ and $\left(\Phi^\tau_{G_\tau}\right)^{-1}$ is its inverse. 
    
    For brevity we denote 
    \begin{align*}
        \Psi^{t-1}(x) &=\left(\Phi^{t-1}_{G_{t-1}}\right)^{-1} \! \left(x\right), \\
        v(s,t) &= \Psi^{t-1}\left(u\left(s,1\right)\right).
    \end{align*}
    Note that 
    \begin{align*}
        \frac{\partial v}{\partial s} &= D\Psi^{t-1}\cdot\frac{\partial u}{\partial s}(s,1), \\
        \frac{\partial v}{\partial t} &= -X_{G_{t-1}}(v(s,t)),
    \end{align*}
    and therefore
    \[
        \omega\left(\frac{\partial v}{\partial s},\frac{\partial v}{\partial t}\right) = \frac{\partial}{\partial s} G_{t-1}\circ\Psi^{t-1}(v(s,1)).
    \]

We decompose the boundary of $\Gamma \cup \Delta$ into four parts: 
\begin{itemize}
    \item $A = \mathbb{R}\times\lbrace 0 \rbrace$,
    \item $B_- = \left((-\infty,0] \times \lbrace 0\rbrace\right)\cup \left\lbrace (s,t)\in \mathbb{R}^2 \,\middle\vert\, 0\le s \le 1, t=s-1 \right\rbrace$,
    \item $B_0 = \lbrace 1 \rbrace \times [1,2]$,
    \item and $B_+ = [1,\infty) \times \lbrace 0\rbrace$.
\end{itemize}
Note that $w$ maps $B_-$ to $\widehat{L}^\prime_0$ and $B_+$ to $\widehat{L}^\prime_{\beta(1)}$ = $\widehat{L}^\prime_T$.
See Figure \ref{fig:gammaDelta}.

\begin{figure}[H]
	\centering
	\includegraphics[scale=0.75]{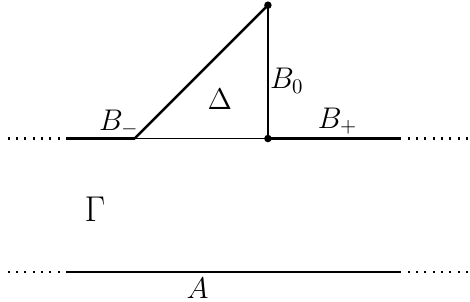}
	\caption{\small{The boundary decomposition of $\Gamma\cup\Delta$.
	}}
	\label{fig:gammaDelta}
\end{figure}

By an application of Stokes theorem, we get
\begin{align*}
    \intop_{\Gamma\cup \Delta} w^*\omega &= \intop_{\gamma_-} \lambda + \intop_{B_-} w^*\lambda + \intop_{B_0} w^*\lambda + \intop_{B_+} w^*\lambda - \intop_{\gamma_+} \lambda + \intop_A w^*\lambda \\
    &= \intop_{\gamma_-} \lambda + \intop_{B_-} w^*\lambda + \intop_{B_0} w^*\lambda + \intop_{B_+} u^*\lambda - \intop_{\gamma_+} \lambda + \intop_A u^*\lambda \\
    &= \intop_{\gamma_-} \lambda + \intop_{B_-} w^*dg_{0} + \intop_{B_0} w^*\lambda + \intop_{B_+} u^*dg^{}_{T} - \intop_{\gamma_+} \lambda + \intop_A u^*df \\
    &=\intop_{\gamma_-} \lambda + g_{0}\left(\Psi^{1}\left(u(1,1)\right)\right) - g_{0}\left(\gamma_-(1)\right) + \intop_{B_0} w^*\lambda + g^{}_{T}\left(\gamma_+(1)\right)  \\ &\phantom{- } - g^{}_{T}\left(u(1,1)\right) - \intop_{\gamma_+} \lambda + f\left(\gamma_-(0)\right) - f\left(\gamma_+(0)\right) \\
    &= \intop_{\gamma_-} \lambda + f\left(\gamma_-(0)\right) - g_{0}\left(\gamma_-(1)\right) - \intop_{\gamma_+} \lambda - f\left(\gamma_+(0)\right) + g^{}_{T}\left(\gamma_+(1)\right) + \\&\phantom{+ } + g_{0}\left(\Psi^{1}\left(u(1,1)\right)\right) + \intop_{B_0} w^*\lambda - g^{}_{T}\left(u(1,1)\right).
\end{align*}

The first six terms in the formula after the last equality also appear in the formula for the actions of $\gamma_-$ and $\gamma_+$, so we are left with the goal of estimating the last three terms.

Since all the generators for the Floer complex lie in $M$, we deduce by the $C^0$-bound, Proposition \ref{prop:c0Cofinement}, that all the Floer solutions are contained in $M$ as well.

Since $\Psi^{1}\left(u(1,1)\right) \in L^\prime_0$, we have that 
\[
    \left\vert g_{0}\left(\Psi^{1}\left(u(1,1)\right)\right)\right\vert \le \max_{L^\prime_0} g_0 =: C_1.
\]
Also, $u(1,1)\in L^\prime_T$, hence
\[
    \left\vert g^{}_{T}\left(u(1,1)\right)  \right\vert \le \max_{L^\prime_T} g^{}_T =: C_2.
\]
The restriction $w\vert^{}_{B_0}$ is a paramterization of a Hamiltonian chord of $G_{t-1}$ for $1\le t \le 2$, hence
\[
    \left\vert \intop_{B_0} w^*\lambda \right\vert \le \max_{x\in L^\prime_0} \intop_0^1 \lambda\left(\dot\Phi^t_{X_{G_t}}(x)\right)=: C_3.
\]

Back to the expression for the energy, by subtracting and adding the integral over $\Delta$, we get
\begin{align*}
E\left( u \right) &= \intop_0^1 H\left( \gamma_+\left(t\right)\right) dt - \intop_0^1 H\left( \gamma_-\left(t\right)\right) dt - \intop_{\Gamma\cup\Delta} u^*\omega + \intop_{\Delta} w^*\omega \\
&=\intop_0^1 H\left( \gamma_+\left(t\right)\right) dt - \intop_0^1 H\left( \gamma_-\left(t\right)\right) dt - \intop_{\gamma_-} \lambda - f\left(\gamma_-(0)\right) + g_{0}\left(\gamma_-(1)\right)\\
&\phantom{+ } + \intop_{\gamma_+} \lambda + f\left(\gamma_+(0)\right) - g_{T}\left(\gamma_+(1)\right) \\
&\phantom{- } -g_{0}\left(\Psi^{1}\left(u(1,1)\right)\right) - \intop_{B_0} w^*\lambda + g_{T}\left(u(1,1)\right) + \intop_{\Delta} w^*\omega\\
&=\mathcal{A}^{L \to L^\prime_T}_H\left(\gamma_+\right) - \mathcal{A}^{L \to L^\prime_0}_H\left(\gamma_+\right) - g_{0}\left(\Psi^{1}\left(u(1,1)\right)\right) - \intop_{B_0} w^*\lambda + g_{T}\left(u(1,1)\right) + \intop_{\Delta} w^*\omega. \\
\end{align*}
Of the last four terms the first three were already estimated and we are left with estimating $\intop_{\Delta} w^*\omega$.
\begin{align*}
    \intop_{\Delta} w^*\omega &= \intop_{\Delta} v^*\omega \\
    &= \intop_0^1 \intop_1^{1+s} \omega\left(\frac{\partial v}{\partial s},\frac{\partial v}{\partial t}\right)dt\,ds \\ 
    &=\intop_1^2 \intop_{t-1}^{1} \frac{\partial}{\partial s} G_{t-1}\circ\Psi^{t-1}(v(s,1))\,ds\,dt \\
    &= \intop_1^2 \left(G_{t-1}\circ\Psi^{t-1}(u(1,1)) - G_{t-1}\circ\Psi^{t-1}(u(t-1,1))\right) dt.
\end{align*}
Hence
\[
    \left\vert \intop_{\Delta} w^*\omega \right\vert \le \max_{M\times [0,1]}G_t(x) - \min_{M\times [0,1]}G_t(x) = \operatorname*{osc}\limits_{M\times [0,1]}G_t(x).
\]

To conclude 
\[E(u) \le \mathcal{A}^{L \to L^\prime_T}_H\left(\gamma_+\right) - \mathcal{A}^{L \to L^\prime_0}_H\left(\gamma_+\right) + \max_{L^\prime_0} g_0 + \max_{L^\prime_T} g_T  + C_3 + \operatorname*{osc}\limits_{M\times [0,1]}G_t(x),\]
which is bounded by the compactness of $M$.
\end{proof}

\section{Proof of Theorem \ref{thm:mainTheorem}}\label{sec:mainTheoremProof}

We will obtain Theorem \ref{thm:mainTheorem} from the following special case, where the parameter $b$ is normalized to $1$:
\begin{theorem}\label{thm:specialCaseMainTheorem}
    In the setting of Theorem \ref{thm:mainTheorem}, assume that the barcode of \\ $\wfhBold{M}{L}{L^\prime}$ contains a bar of the form $\left (\mu, C\mu\right]$ with $\mu > 0$, $C>1$ or a bar of the form $\left(\mu, \infty\right)$, In the latter case we set $C=\infty$. Then, for every $0<\eps <1$ with $\frac {1}{\varepsilon} \le C$, such that $L$ and $L^\prime$ are conical in $M_{[\eps,1]}$, we have that:
                \[
                \operatorname{pb}_{\widehat {M}}^+\left(\widehat{L}\cap M_{[\eps,1]},\widehat{L^\prime}\cap M_{[\eps,1]}, \partial M_\eps, \partial M \right) \ge \frac {1}{\mu\left(1-\eps\right)}.
                \]
\end{theorem}

We first prove that Theorem \ref{thm:specialCaseMainTheorem} implies Theorem \ref{thm:mainTheorem}.
\begin{proof}[Proof of Theorem \ref{thm:mainTheorem}]

Denoting by $\varphi$ the flow of the Liouville vector field, we note that for all $a>0$ we have that $M_a = \varphi^{\log a}\left(M\right)$. Since $M_b$ is by itself a Liouville domain we can apply the same construction to it, noting that 
\[
\left(M_b\right)_a = \varphi^{\log a}\left(M_b\right) = \varphi^{\log a}\left(\varphi^{\log b}\left(M\right)\right) = \varphi^{\log a+\log b}\left(M\right) =  \varphi^{\log ab}\left(M\right)=M_{ab}.
\]

Thus, to obtain Theorem \ref{thm:mainTheorem} from Theorem \ref{thm:specialCaseMainTheorem}, we first note that since $\left(M_b\right)_{\frac{1}{b}} = M$, Lemma \ref{lem:viterboRestriction}, Item 1, implies that for all $\mu$:
\[\operatorname{HW}^\mu\left(M,L \to L^\prime \right) \cong \operatorname{HW}^{b\mu}\left(M_b,\,\widehat{L}\cap M_b \to \widehat{L^\prime} \cap M_b\right),
\]
hence, a bar of the form $(\mu, C\mu]$ in $\wfhBold{M}{L}{L^\prime}$, gives rise to a bar of the form $\left(b\mu, Cb\mu\right]$ in $\wfhBold{M_b}{\widehat{L}\cap M_b}{\widehat{L^\prime} \cap M_b}$.

Now, applying Theorem \ref{thm:specialCaseMainTheorem} to $M_b$, $\widehat{L}\cap M_b$ and $\widehat{L^\prime} \cap M_b$, with $\eps=\frac{a}{b}$, yields 
\[
    \operatorname{pb}_{\widehat {M}}^+\left(\widehat{L} \cap \left(M_b\right)_{[\frac{a}{b},1]},\widehat{L^\prime} \cap \left(M_b\right)_{[\frac{a}{b},1]}, \partial \left(M_b\right)_{\frac{a}{b}}, \partial M_b \right) \ge \frac {1}{b\mu\left(1-\frac{a}{b}\right)}=\frac {1}{\mu\left(b-a\right)}.
\]
By the construction explained above, 
\begin{multline*}
    \operatorname{pb}_{\widehat {M}}^+\left(\widehat{L} \cap \left(M_b\right)_{[\frac{a}{b},1]},\widehat{L^\prime} \cap \left(M_b\right)_{[\frac{a}{b},1]}, \partial \left(M_b\right)_{\frac{a}{b}},\, \partial M_b \right) = \\ = \operatorname{pb}_{\widehat {M}}^+\left(\widehat{L} \cap M_{[a,b]},\,  \widehat{L^\prime} \cap M_{[a,b]},\,  \partial M_a,\, \partial M_b \right),
\end{multline*}
and this concludes the proof of Item 1.

We deduce Item 2 in the theorem from Item 1 and the semi-continuity of $\operatorname{pb}^+$. Since in the limit $a\to 0$, the sets $M_{[a,b]}$ Hausdorff converge to $M_b$, and $\partial M_a$ converges to $\operatorname{core}(M)$, Semi-continuity implies that 
 \begin{multline*}  
        \operatorname{pb}_{\widehat {M}}^+\left(\widehat{L}\cap M_{b},\widehat{L^\prime}\cap M_{b}, \operatorname{core}(M), \partial M_b \right) \ge \\
        \ge \limsup_{a\to 0} \operatorname{pb}_{\widehat {M}}^+\left(\widehat{L}\cap M_{[a,b]},\widehat{L^\prime}\cap M_{[a,b]}, \partial M_{a}, \partial M \right) \ge \\
        \ge \lim_{a\to 0} \frac {1}{\mu \left(b-a\right)} = \frac {1}{\mu b }.
\end{multline*}
\end{proof}

\subsection{\texorpdfstring{$\operatorname{pb}^+_M$}{pb+} and Deformations of the Symplectic Form.}
The typical approach for establishing a lower bound on $pb^+$ involves finding obstructions for suitable deformations of the symplectic form. (see e.g \cite[proof of Proposition 5.1]{entov2017lagrangian}.)
\begin{lemma}\label{lem:pb4DeformationBound}
    Let $X_0,X_1,Y_0,Y_1$ be an admissible quadruple, suppose there exists a number $\kappa>0$, such that for all pairs of admissible functions $(F,G)\in \mathcal{G}_M\left(X_0,X_1,Y_0,Y_1\right)$, there exists $0<t\le \kappa$ such that the form
    \[
    \omega_t := \omega + tdF\wedge dG
    \]
    is \emph{degenerate}. Then,
    \[
        \operatorname{pb}^+_M\left(X_0,X_1,Y_0,Y_1\right) \ge \frac{1}{\kappa}
    \]
\end{lemma}
\begin{proof}
    Following \cite{entov2017lagrangian, buhovsky2012poisson}, consider the deformation
    \[
    \omega_\tau := \omega + \tau dF\wedge dG
    \]
    A calculation shows that
    \[
        dF\wedge dG \wedge \omega^{\wedge (n-1)} = -\frac{1}{n}\lbrace F,G \rbrace\omega^{\wedge n}.
    \]
    Therefore
    \[
    \omega_\tau^{\wedge n} = \left(1 - \tau \left\lbrace F,G \right\rbrace \right)\omega^{\wedge n},
    \]
    and hence, for all $\tau\in \left[0,\frac{1}{\max_M \left\lbrace F,G \right\rbrace}\right)$, $\omega_\tau$ is non-degenerate.
    By the assumption of the lemma, there exists $0<t\le\kappa$ such that $\omega_t$ is degenerate, hence $\kappa \ge t \ge \frac{1}{\max_M \left\lbrace F,G \right\rbrace}$, and since $\kappa$ is independent of $F$ and $G$, the result follows.
\end{proof}

\subsection{Reformulation with an Equivalent Quadruple}

For the sake of convenience in the proof, instead of working directly with $\operatorname{pb}_{\widehat {M}}^+\left(\widehat{L}\cap M_{[\eps,1]},\widehat{L^\prime}\cap M_{[\eps,1]}, \partial M_\eps, \partial M \right)$ we will show it is equivalent to a Poisson bracket invariant of another  quadruple, which is $\operatorname{pb}^+_{\widehat{M}} \left(\widehat{L}, \widehat{L^\prime},  \partial M_\varepsilon, M_{[1,\infty)} \right)$, to be defined below.

By a slight abuse of notation\footnote{Within the scope of this paper, we defined the sets $\mathcal{G}$ and the $\operatorname{pb}^+$ invariant for a quadruple of compact sets. Here we extend the notation for special cases with non-compact sets. See \cite{entov2022legendrian} for a more general definition for the non-compact case.} we denote the following set of pairs of functions
\[
\left\lbrace \left(F,G\right) \in C^\infty_c(\widehat{M}) \times C^\infty(\widehat{M})  \,\middle\vert\,
\begin{aligned}
    F\vert^{}_{\mathcal{O}p\left(L\cap  M_{[\varepsilon,1]}\right)} \equiv 0, &&
    F\vert^{}_{\mathcal{O}p\left(L^\prime \cap  M_{[\varepsilon,1]}\right)} \equiv 1 \\
    G\vert^{}_{\mathcal{O}p\left( M_\varepsilon\right)} \equiv 0, && G\vert^{}_{\mathcal{O}p\left(M_{[1,\infty)}\right)} \equiv 1 \\
\end{aligned}
\right\rbrace,
\]
by $\mathcal{G}_{\widehat{M}} \left(L \cap  M_{[\varepsilon,1]}, L^\prime \cap  M_{[\varepsilon,1]},   M_\varepsilon, M_{[1,\infty)} \right)$, and the set 
\[
\left\lbrace \left(F,G\right) \in C^\infty(\widehat{M}) \times C^\infty(\widehat{M})  \,\middle\vert\,
\begin{aligned}
    F\vert^{}_{\mathcal{O}p\left(\widehat{L}\right)} \equiv 0, &&
    F\vert^{}_{\mathcal{O}p\left(\widehat{L^\prime} \right)} \equiv 1 \\
    G\vert^{}_{\mathcal{O}p\left(M_\varepsilon\right)} \equiv 0, && G\vert^{}_{\mathcal{O}p\left(M_{[1,\infty)}\right)} \equiv 1 \\
\end{aligned}
\right\rbrace.
\]
by $\mathcal{G}_{\widehat{M}} \left(\widehat{L} , \widehat{L^\prime},  M_\varepsilon, M_{[1,\infty)} \right)$.

Note that while $G$ and $F$ are not compactly supported, the conditions imply that $dG$ is compactly supported, hence $\left\lbrace F,G \right\rbrace$ is compactly supported, hence obtains its maximum.

Proceeding with the abuse of notations, we then define 
\begin{multline*}
    \operatorname{pb}^+_{\widehat{M}} \left(L \cap  M_{[\varepsilon,1]}, L^\prime \cap  M_{[\varepsilon,1]},  \partial M_\varepsilon, M_{[1,\infty)} \right) := \\
     \inf_{\mathcal{G}_{\widehat{M}} \left(L \cap  M_{[\varepsilon,1]}, L^\prime \cap  M_{[\varepsilon,1]},  \partial M_\varepsilon, M_{[1,\infty)} \right)} \max_{\widehat{M}} \left\lbrace F,G \right\rbrace,
\end{multline*}
allowing for one of the four sets to be the whole end at infinity, and
\[
    \operatorname{pb}^+_{\widehat{M}} \left(\widehat{L}, \widehat{L^\prime},  \partial M_\varepsilon, M_{[1,\infty)} \right) :=
     \inf_{\mathcal{G}_{\widehat{M}} \left(\widehat{L} , \widehat{L^\prime},  \partial M_\varepsilon, M_{[1,\infty)} \right)} \max_{\widehat{M}} \left\lbrace F,G \right\rbrace,
\]
also allowing for the whole conical Lagrangians into the quadruple.
We then claim that $\operatorname{pb}^+_{\widehat{M}} \left(\widehat{L}, \widehat{L^\prime},  \partial M_\varepsilon, M_{[1,\infty)} \right)$ equals the invariant we are interested in:
\begin{claim}\label{clm:pbConvenientForm}
    \begin{multline*}
        \operatorname{pb}_{\widehat {M}}^+\left(L \cap  M_{[\varepsilon,1]}, L^\prime \cap  M_{[\varepsilon,1]}, \partial M_\eps, \partial M \right) = \\
        =
        \operatorname{pb}^+_{\widehat{M}} \left(L \cap  M_{[\varepsilon,1]}, L^\prime \cap  M_{[\varepsilon,1]},  M_\varepsilon, M_{[1,\infty)} \right) = \\
        = 
        \operatorname{pb}^+_{\widehat{M}} \left(\widehat{L}, \widehat{L^\prime},   M_\varepsilon, M_{[1,\infty)} \right).
    \end{multline*} 
\end{claim}
\begin{proof}
We first prove the first equality, by proving two inequalities.

Starting with the $\le$ direction: \\
For every pair $(F,G)\in \mathcal{G}_{\widehat{M}} \left(L \cap  M_{[\varepsilon,1]}, L^\prime \cap  M_{[\varepsilon,1]},  \partial M_\varepsilon, \partial M \right)$, the function $G$ is equal to a constant, $1$, near $\partial M$, which is a set separating $\widehat{M}$ into two connected components. Similarly $G\equiv 0$ near $\partial M_\varepsilon$, also a set separating $\widehat{M}$ into two connected components.
Hence, one can replace $G$ by
\[
\overline{G}\left(x\right) := \begin{cases}
1 & x \in M_{[1,\infty)} \\
G(x) & x \in M_{[\varepsilon,1]} \\
0 & x \in M_\varepsilon
\end{cases},
\]
which is also smooth, and therefore 
\[(F,\overline{G})\in \mathcal{G}_{\widehat{M}} \left(L \cap  M_{[\varepsilon,1]}, L^\prime \cap  M_{[\varepsilon,1]},  M_\varepsilon, M_{[1,\infty)} \right).\]

Now, since $d\overline{G}=0$ in $M_{\eps+\delta} \cup M_{[1-\delta,\infty)}$, for some small $\delta>0$, it holds that $\left\lbrace F, \overline{G} \right\rbrace = 0$ there, therefore one has that  $\max_{\widehat{M}}\left\lbrace F, G \right\rbrace \ge \max_{\widehat{M}}\left\lbrace F, \overline{G} \right\rbrace$, proving the inequality.

In the other direction ($\ge$):\\
Take $(F,G)\in \mathcal{G}_{\widehat{M}} \left(L \cap  M_{[\varepsilon,1]}, L^\prime \cap  M_{[\varepsilon,1]},  M_\varepsilon,  M_{[1,\infty)} \right)$
The function $F$ is compactly supported, hence there exists $K>1$ such that $\operatorname{supp} F \subset M_K$. Let $\rho\colon \widehat{M} \to \mathbb{R}$ be a cutoff function such that $\rho\vert^{}_{M_K} \equiv 1$ and $\rho\vert^{}_{M_{[2K,\infty)}} \equiv 0$.
Then the pair $(F,\rho G)$ is in $\mathcal{G}_{\widehat{M}} \left(L \cap  M_{[\varepsilon,1]}, L^\prime \cap  M_{[\varepsilon,1]}, \partial M_\varepsilon,  \partial M \right)$.
Since $F\vert^{}_{[K,\infty)}\equiv 0$, and $\rho G\vert^{}_{M_K} = G\vert^{}_{M_K}$, we have that 
\[
\max_{\widehat{M}}\left\lbrace F, \rho G \right\rbrace \ge \max_{\widehat{M}}\left\lbrace F, G\right\rbrace, 
\]
completing the proof.

The second equality follows from the fact that if $G$ satisfies that there exists a $\delta > 0$ such that it is locally constant in $M_{\varepsilon + \delta} \cup M_{[1-\delta,\infty)}$, and hence $dG=0$ there, therefore it follows that $\lbrace F,G \rbrace=0$ in $M_{\varepsilon + \delta} \cup M_{[1-\delta,\infty)}$ for every $F\in C^\infty(\widehat{M})$. Hence, for every two functions $F_1,F_2 \in C^\infty(\widehat{M})$ such that $F_1\vert^{}_{M_{[\varepsilon,1]}} = F_2\vert^{}_{M_{[\varepsilon,1]}}$ it holds that 
\[
\max_{\widehat{M}} \lbrace F_1,G \rbrace = \max_{\widehat{M}} \lbrace F_2,G \rbrace.
\]
\end{proof}

\subsection{A Lagrangian Isotopy from Admissible Pairs.}\label{ssec:lagrangianIsotopy}
Let $L_0, L^\prime_0$ be two asymptotically conical exact Lagrangians in $M$, with $f,g$ primitives of the restrictions of $\lambda$ to $L_0$ and $ L^\prime_0$, respectively, chosen such that they restrict to a constant, $0$, in a neighborhood of $\partial M$. Moreover, assume that  $L_0$ and $L^\prime_0$ are conical in $M_{[\eps,1]}$. Let $(F,G)\in\mathcal{G}_{\widehat{M}} \left(\widehat{L} , \widehat{L^\prime} ,M_\varepsilon, M_{[1,\infty)} \right)$ be a pair of admissible functions.
In this section we construct from $(F,G)$ a Lagrangian isotopy which will play a role in our proof of theorem \ref{thm:mainTheorem}, and study some Floer theoretical constructions related to it.

Consider the following deformation of $\omega$ in $\widehat{M}$:
\[
\omega_\tau = \omega + \tau dF\wedge dG.
\]
Note that the form $dF\wedge dG$ is compactly supported.

Our first step is to convert this deformation to a Lagrangian isotopy with respect to a fixed symplectic form.

Note that $\tau dF\wedge dG = \tau d\left(F dG\right)$, hence $\frac{d}{d\tau} \omega_\tau$ is exact with a a compactly supported primitive, $F dG$. Moreover $\omega_\tau = d\left(\lambda + \tau FdG\right)$, with the primitives $\lambda_\tau = \lambda + \tau FdG$ all coinciding with each other, in the complement of a compact set.
Therefore, for every $T>0$ we can apply Moser's stability theorem to the deformation $\left\lbrace\omega_\tau\right\rbrace_{0\le \tau \le T}$. We recall it:
\begin{theorem}[Moser's stability theorem {\cite[Theorem 6.8]{cieliebak2012stein}}] \label{thm:moserTrick}

Let $V$ be a manifold (without boundary but not necessarily compact). Let $\omega_t$, $t \in [0, 1]$, be a smooth family of symplectic forms on $V$ which coincide outside a compact set and such that the cohomology class with compact support $[\omega_t -\omega_0] \in H^2_c(V ;\mathbb{R})$ is independent of $t$. Then there exists a diffeotopy $\phi_t$ with $\phi_t = \operatorname{Id}$ outside a compact set such that $\phi^*_t\omega_t=\omega_0$.

In particular, this applies if $\omega_t = d\lambda_t$ for a smooth family of 1-forms $\lambda_t$ which coincide outside a compact set, and in this case there exists a smooth family of
functions $f_t \colon V \to \mathbb{R}$ with compact support such that
\[
\phi^*_t\lambda_t - \lambda_0 = df_t.
\]\qed
\end{theorem}

We thus deduce that there exists a diffeotopy $\phi_\tau$ with $\phi^*\omega_\tau = \omega$, and a family of smooth functions which we denote by $h_\tau\colon \widehat{M} \to \mathbb{R}$ such that $\phi^*_\tau\lambda_\tau - \lambda = dh_\tau$.

Define Lagrangians $\widehat{L}_\tau := \phi_\tau^{-1}\left(\widehat{L}\right)$ and $\widehat{L^\prime}_\tau := \phi_\tau^{-1}\left(\widehat{L^\prime}\right)$.
\begin{proposition}\label{prop:lagrangianIsotopyProperties}
     The following properties hold:
    \begin{enumerate}
        \item $\widehat{L}_\tau$ = $L$ for all $\tau$.
        \item The deformation $\widehat{L^\prime}_\tau$ is supported in the interior of $M_{[\varepsilon,1]}$, namely, there exists $\delta > 0$ such that for all $\tau$:
        \[\widehat{L^\prime}_\tau \cap M_{\varepsilon+\delta} = \widehat{L^\prime} \cap M_{\varepsilon+\delta},\text{ and } \,\widehat{L^\prime}_\tau \cap M_{[1-\delta,\infty)} = \widehat{L^\prime} \cap M_{[1-\delta,\infty)}.
        \]
        \item $\widehat{L}_\tau$ and $\widehat{L^\prime}_\tau$ are both exact Lagrangians. Moreover $\lambda\vert^{}_{\widehat{L}_\tau} = df$, and there exist functions $g_\tau$ such that $\lambda\vert^{}_{\widehat{L}^\prime_\tau} = dg_\tau$. Each function $g_\tau$ vanishes in a neighborhood of $\widehat{L}_\tau \cap M_\varepsilon$, and $g_\tau = \tau$ in a neighborhood of $\widehat{L}_\tau \cap M_{[1,\infty)} $.
    \end{enumerate}
\end{proposition}
\begin{proof}
    Let us now analyze the behaviour of $\phi_\tau$ and $f_\tau$ along the sets $\widehat{L}$, $\widehat{L^\prime}$, $M_\varepsilon$ and $M_{[1,\infty)}$.
    
    Denote $\lambda_\tau = \lambda + \tau FdG$. Then $\omega_\tau = d\lambda_\tau$.
    From the proof in \cite{cieliebak2012stein} it follows that $\phi_\tau$ is the flow of the vector field $X_\tau$ solving
    \[
    \frac{d}{d\tau}\lambda_\tau + \iota^{}_{X_\tau}\omega_\tau = 0.
    \]
    The derivative, $\frac{d}{d\tau}\lambda_\tau = FdG$, vanishes in neighborhoods of $L$,  $ M_\varepsilon$ and $M_{[1,\infty)}$. That is because $F=0$ in a neighborhood of $L$ and $dG=0$ in neighborhoods of $M_\eps$ and $M_{[1,\infty)}$. Therefore, $X_\tau = 0$ in these neighborhoods, and $\phi_\tau = id$ there, which proves Items 1 and 2.
    We now prove Item 3. We first recall that the proof in \cite{cieliebak2012stein} gives the following formula for $h_\tau$:
    \[
        h_\tau = \intop_0^\tau \iota_{X_s}\lambda_s ds.
    \]
    It follows that $h_\tau \equiv 0$ in neighborhoods of  $M_\varepsilon$ and $M_{[1,\infty)}$, 
    
    Since $F\equiv1$ near $\widehat{L^\prime}$, we have that  $\lambda_\tau \vert^{}_{\widehat{L}^\prime} = dg + \tau dG = \left(g + \tau G\right)$; we also note that $\lambda_\tau \vert^{}_{\widehat{L}} = \lambda \vert^{}_{\widehat{L}} = df$, because near $\widehat{L}_\tau$, the form $dG=0$. Set $\tilde{g}_\tau = g + \tau G$. By Theorem \ref{thm:moserTrick}:
    \[
     \phi^*_\tau\lambda_\tau - \lambda = dh_\tau,
    \]
    hence
    \[
     \lambda_\tau - \phi^{-1*}_\tau\lambda = {\phi^{-1}_\tau}^* dh_\tau = d\left(h_\tau \circ \phi^{-1}_\tau\right),
    \]
    so
    \[
     \phi^{-1*}_\tau\lambda = \lambda_\tau - d\left(h_\tau \circ \phi^{-1}_\tau\right).
    \]
    We compute:
    \begin{multline*}
        \lambda\vert^{}_{\widehat{L}^\prime_\tau} = \lambda\vert^{}_{\phi^{-1}_\tau\left(\widehat{L}^\prime\right)} = 
        \phi^{-1*}_\tau\lambda\vert^{}_{\widehat{L}^\prime} = \lambda_\tau\vert^{}_{\widehat{L}^\prime} -d\left(h_\tau \circ \phi^{-1}_\tau\right)\vert^{}_{\widehat{L}^\prime} = \\ = 
        d\tilde{g}_\tau\vert^{}_{\widehat{L}^\prime} -d\left(h_\tau \circ \phi^{-1}_\tau\right)\vert^{}_{\widehat{L}^\prime} = 
        d\left( \tilde{g}_\tau - h_\tau \circ \phi^{-1}_\tau\right)\vert^{}_{\widehat{L}^\prime},
    \end{multline*}
     and set $g_\tau =  \tilde{g}_\tau - h_\tau \circ \phi^{-1}_\tau = g + \tau G  - h_\tau \circ \phi^{-1}_\tau $. The above computation shows that indeed $\lambda\vert^{}_{\widehat{L}^\prime_\tau} = dg_\tau$.

     Since $G\equiv 0$ in a neighborhood of $M_\eps$ and $G\equiv 1$ in a neighborhood of $M_{[1,\infty)}$, we have that  $\tilde{g}_\tau = g$ near $M_\eps \cap\widehat{L^\prime}_\tau$ and $\tilde{g}_\tau = g+\tau$ near $M_{[1,\infty)}\cap\widehat{L^\prime}_\tau$.
    
    Now, as we have shown, in a neighborhood of $M_\varepsilon$ the diffeotopy $\phi_\tau$ is the identity, the function $h_\tau = 0$ and $\tilde{g}_\tau = 0$, hence hence $g_\tau = 0$, and in a neighborhood of $M_{[1,\infty)}$, we also have that $\phi_\tau = id$, that $h_\tau = 0$ and that $\tilde{g}_\tau = 0$, hence $g_\tau = \tau$.  Concluding the proof of item 3.
\end{proof}

The isotopy $\widehat{L^\prime}_\tau$ is an isotopy of exact Lagrangians with respect to the same primitive $\lambda$ of $\omega$. By \cite[Exercise 11.3.17]{mcduff2017introduction} there exists a Hamiltonian isotopy $\psi_\tau\colon M \to M$, such that $\widehat{L^\prime}_\tau = \psi \left( \widehat{L^\prime} \right)$. Since for all $\tau$, $\widehat{L^\prime}_\tau \cap L = \emptyset$, the isotopy $\psi_\tau$ can be chosen to be identity in a neighborhood of $\widehat{L}$.



The next two propositions study how this isotopy interacts with certain wrapped Floer homologies.
\begin{proposition}\label{prop:conicalCommDiagram}
    Let $L, L^\prime$ be two disjoint exact Lagrangians in a Livouille domain $M=\left(Y,\omega,\lambda\right)$, such that $L$ and $L^\prime$ are conical in $M_{[\eps,1]}$. 
    Assume there exists an isotopy of exact Lagrangians, $L^\prime_\tau$ for $0\le\tau\le T$, starting at $L^\prime_0 = L^\prime$, with $\lambda\vert^{}_{L^\prime}=dg_\tau$ for a family of functions $g_\tau \colon L^\prime _\tau \to \mathbb{R}$.
    Assume that the following conditions are satisfied:
    \begin{itemize}
        \item In a neighborhood $\mathcal{U}$ of $\partial M_\eps$, $L^\prime_\tau$ is conical and  $g_\tau\vert^{}_{L^\prime_\tau\cap \mathcal{U}}\equiv 0$ for all $\tau \in [0,T]$,
        \item and in a neighborhood $\mathcal{V}$ of $\partial M$,  $L^\prime_\tau$ is conical and  $g_\tau\vert^{}_{L^\prime_\tau\cap \mathcal{V}}\equiv \tau$ for all $\tau \in [0,T]$.
    \end{itemize}
    Then, for every $\mu \ge 0$, the following diagram is well defined and commutes:
    \begin{equation*} 
        \begin{tikzcd}
        {HW^{\mu}\left(M,L \to L^\prime_0\right)} \arrow[dd, "{\mathcal{R}es}^{}_{M\to M_\varepsilon}"] \arrow[rr, "\mathcal{C}ont_{\{L^\prime_\tau\}_{0}^{T}}"] &  & {HW^{\mu-T}\left(M,L \to L^\prime_T\right)} \arrow[d, "{\mathcal{R}es}^{}_{M\to M_\varepsilon}"]                                                                \\
                                                                                                                                                                                                      &  & {HW^{\mu-T}\left(M_\varepsilon,L \cap M_\varepsilon\to L^\prime_T\cap M_\varepsilon\right)}                                                                        \\
        {HW^{\mu}\left(M_\varepsilon,L \cap M_\varepsilon\to L^\prime_0\cap M_\varepsilon\right)}                                                                                                     &  & {HW^{\mu-T}\left(M_\varepsilon,L \cap M_\varepsilon\to L^\prime_0\cap M_\varepsilon\right)} \arrow[ll, "\iota^{}_{(\mu-T)\to\mu}"'] \arrow[u, equal]
        \end{tikzcd}
        \end{equation*}
    The action filtration is defined with respect to the functions $g_\tau\colon L^\prime_\tau \to \mathbb{R}$ and a function $f\colon L\to \mathbb{R}$ which is taken to vanish in $M_{[\eps,1]}$. Such $f$ exists because $L$ is conical.
    The morphism $\mathcal{C}ont$ is defined in subsection \ref{ssec:continuationMovingBdry} and the morphism ${\mathcal{R}es}$ is defined in subsection \ref{ssec:viterboTransfer}.
\end{proposition}
\begin{proof}
    Since $M_{[\eps,1]}$ is a trivial cobordism, and the Lagrangians are conical therein, it holds that $\widehat{M} = \widehat{M_\eps}$ and that also $\widehat{L} = \widehat{L\cap M_\eps}$ and $\widehat{L^\prime_\tau} = \widehat{L^\prime_\tau\cap M_\eps}$ for all $\tau$.
    
    For a $\mu$, non-characteristic for both $M$ and $M_\eps$, we choose admissible Hamiltonians $H^\mu$ and $G^\mu$ on $\widehat{M}$ and $\widehat{M_\eps}$ such that 
    \begin{align*}
        H\left( CF\left( H^\mu, \widehat{L} \to \widehat{L^\prime_0} \right) \right) = HW\left( H^\mu, L \to L^\prime_0 \right) = HW^\mu \left( M, L \to L^\prime_0 \right), \\
        H\left( CF\left( G^\mu, \widehat{L} \to \widehat{L^\prime_0} \right) \right) = HW\left( G^\mu, L \to L^\prime_0 \right) = HW^\mu \left( M_\eps, L \to L^\prime_0 \right). 
    \end{align*}

The left equalities hold by definition, and the right ones are by proposition \ref{prop:actionAndSlope}, together with the fact that the slope $\mu$ is non-characteristic. Also note that $\widehat{M} = \widehat{M_\eps}$

By the admissibility of $H^\mu$ and $G^\mu$, they are of the form $\mu r + b_{H^\mu}$ and $\mu r^\prime + b_{G^\mu}$, respectively, where $r$ and $r^\prime$ are radial coordinates near $\partial M$ and $\partial M_\eps$, respectively, and $b_{H^\mu}$ and $b_{G^\mu}$ are some negative constants.

The constants $b_{H^\mu}$ and $b_{G^\mu}$ can be chosen to be small enough in their absolute value, such tht the full complex equals the $\mu$ action truncated one, namely:
\[CF\left( H^\mu, \widehat{L} \to \widehat{L^\prime_0} \right) = CF^\mu\left( H^\mu, \widehat{L} \to \widehat{L^\prime_0}\right),\] 
and 
\[CF\left( G^\mu, \widehat{L} \to \widehat{L^\prime_0} \right) = CF^\mu\left( G^\mu, \widehat{L} \to \widehat{L^\prime_0}\right).\]
Indeed, the generators appear near the boundary, the spectrum is discrete and the action is given by the $y$-intercept of the tangent to the radial Hamiltonian, at points where the slope belongs to the spectrum. Since $\mu$ is noncharacteristic, namely, does not belong to the spectrum, it follows that for constants $b_{H^\mu}$ and $b_{G^\mu}$, small enough in absolute value, the complexes $CF^\nu$ stabilize, as $\nu \to \mu$.

We now consider the following diagram of complexes.
\begin{equation}\label{eqn:commDiagSquareContRes2}
%
\begin{tikzcd}
{CW^{\mu}\left(H^\mu,\widehat{L} \to \widehat{L^\prime_0}\right)} \arrow[rr, "\mathcal{C}ont_{\{L^\prime_\tau\}_{0}^{T}}"] \arrow[d, "{\iota^{H^\mu,G^\mu}}"] &  & {CW^{\mu-T}\left(H^\mu,\widehat{L} \to \widehat{L^\prime_T}\right)} \arrow[d, "{\iota^{H^\mu,G^\mu}}"]                      \\
{CW^{\mu}\left({G}^\mu,\widehat{L} \to \widehat{L^\prime_0}\right)} \arrow[d, equal]                                                            &  & {CW^{\mu-T}\left({G}^\mu,\widehat{L} \to \widehat{L^\prime_0}\right)} \arrow[d, equal] \arrow[ll, hook']      \\
{CW^{\mu}\left({G}^\mu,\widehat{L \cap M_\varepsilon}\to \widehat{L^\prime_0\cap M_\varepsilon}\right)}                                                       &  & {CW^{\mu-T}\left({G}^\mu,\widehat{L \cap M_\varepsilon}\to \widehat{L^\prime_0\cap M_\varepsilon}\right)} \arrow[ll, hook']
\end{tikzcd}
\end{equation}
\end{proof}

Recall that in defining $\iota^{H^\mu,G^\mu}$ we choose a generic increasing homotopy through admissible Hamiltonians, $\left(H_s\right)_{s\in\mathbb{R}}$, with $\partial_s H_s \ge 0$ and $H_s = H^\mu$ near $s=-\infty$, and $H_s = G^\mu$ near $s=\infty$, and count the associated continuation solutions. The map $\mathcal{C}ont_{\{L^\prime_\tau\}_{0}^{T}}$ as in \ref{ssec:continuationMovingBdry} involves a choice of a smooth function $\beta(\tau) \colon \mathbb{R} \to \mathbb{R}$ such that $\beta \equiv 0$ for $\tau\in (-\infty,0]$, $\beta \equiv T$ for $\tau \in [1,\infty)$, and in $[0,1]$, $\beta(\tau)$ is monotone increasing, mapping the interval $[0,1]$ onto $[0,T]$. The isotopy $\lbrace L^\prime_{\beta(\tau)} \rbrace_{\tau \in \mathbb{R}}$, is given by $L^\prime_{\beta(\tau)} = \phi^\tau_{G_\tau} \left( L^\prime_0 \right)$. For a Hamiltonian $G_\tau \colon M\times [0,1] \to \mathbb{R}$, which is assumed to be supported away from the boundary of $M$. Note that $G^\mu$ and $G_\tau$ are different Hamiltonians.

We will build a chain homotopy that makes the upper square commute in homology.

By gluing, (see e.g. \cite[Proposition 19.5.1]{oh2015symplectic}, for an analogous result in the context of continuation maps in Hamiltonian Floer homology), the composition $\iota^{H^\mu,G^\mu} \circ \mathcal{C}ont_{\{L^\prime_\tau\}_{0}^{T}}$ equals a map $\Psi$ defined as follows: For a choice of some $R>0$ large enough, which also can be taken large enough so that $H_{s+R} = G^\mu$ for $s \ge 0$, gluing provides us with a map $\Psi$ given by 
\[
    \Psi x \quad = \sum_{x^\prime \in {\scalebox{0.9}{$ \mathcal{P}$}}^H_{\widehat{L}\to \widehat{L^\prime_T}}} \#_2\mathcal{M}^0\left(x,x^\prime, H_{s+R}, J, \lbrace \widehat{L^\prime}^{}_{\!\!\beta(\tau)} \rbrace_{\tau\in\mathbb{R}} \right) \cdot x^\prime,
\]
where, for a given homotopy of Hamiltonians $H_s$, the space  \[\mathcal{M}^0\left(\gamma_-,\gamma_+, H_{s}, J, \lbrace \widehat{L^\prime}^{}_{\!\!\beta(\tau)} \rbrace_{\tau\in\mathbb{R}} \right)\] was defined by equations (\ref{eqn:continuationMovingBoundaryNotation}) and (\ref{eqn:continuationMovingBoundary})

Note that for every $\theta>0$, the Lagrangian isotopy $\lbrace L^\prime_{\theta\beta(\tau)} \rbrace_{\tau \in \mathbb{R}}$, is given by $L^\prime_{\theta\beta(\tau)} = \phi^{\beta^{-1}(\theta\beta(\tau))}_{G_{\beta^{-1}(\theta\beta(\tau))}} \left( L^\prime_0 \right)$, hence, is also generated by the flow of a Hamiltonian.
For $\theta \in [0,1]$, the isotopies $\lbrace L^\prime_{\theta\beta(\tau)} \rbrace_{\tau \in \mathbb{R}}$ smoothly interpolate between the constant isotopy $\lbrace L^\prime_{0} \rbrace_{\tau \in \mathbb{R}}$ and $\lbrace L^\prime_{\beta(\tau)} \rbrace_{\tau \in \mathbb{R}}$.

We define the space  $\mathcal{M}^{-1}\left(\gamma_-,\gamma_+, H_{s}, J, \lbrace \widehat{L^\prime}^{}_{\!\!\theta\beta(\tau)} \rbrace_{\tau\in\mathbb{R},\theta\in[0,1]} \right)$ to be the moduli space of pairs $(u,\theta)$ where $u$ is a Floer continuation solution with moving boundary conditions, of Fredholm index $-1$ (i.e. as a pair $(u,\theta)$ it has index $0$), connecting $\gamma_-$ and $\gamma_+$, solving:
\begin{equation}\label{eqn:chainhomotopyMovingBoundary2}
    \begin{aligned}
        &u\colon \mathbb{R}\times [0,1] \to \widehat{M}, \\
        &\overline{\partial}^{}_{J,H_{s,\theta}}u := \partial_s u + J\left( \partial_t u - X_{H_{s}}\left(u\right) \right) = 0, \\
        &\lim_{s\to\pm \infty} u(s,t) = \gamma_{\pm}(t), \\
        &u(s,0)\in \widehat{L} \text{ and }  u(s,1)\in \widehat{L^\prime}^{}_{\!\!\theta\beta(s)}.
    \end{aligned}  
\end{equation}

Note that for the space to be well defined, $\gamma_+$ must belong to $\mathcal{P}_{L\to L^\prime_t}$ for all $t\in[0,T]$. 

Set 
\[
    \mathfrak{H}\, x \quad = \sum_{x^\prime \in \mathcal{P}^H_{\widehat{L}\to \widehat{L^\prime_T}}} \#_2\mathcal{M}^{-1}\left(\gamma_-,\gamma_+, H_{s+R}, J, \lbrace \widehat{L^\prime}^{}_{\!\!\theta\beta(\tau)} \rbrace_{\tau\in\mathbb{R}} \right) \cdot x^\prime.
\]
Note that all the generators of the complex for $G^\mu$ are chords lying in $M_\eps$, where $L^\prime_t \cap M_\eps = L^\prime_0 \cap M_\eps$ for all $t\in[0,1]$. Hence, $\mathfrak{H}$ is well defined.

Now define
 $\overline{\mathcal{M}^{0}\left(\gamma_-,\gamma_+, H_{s+R}, J, \lbrace \widehat{L^\prime}^{}_{\!\!\theta\beta(\tau)} \rbrace_{\tau\in\mathbb{R}} \right)}$, which is the closure of the moduli space of index $1$ pairs $(u,\theta)$, defined analogously to the moduli space above. I.e. $u$ is an index $0$ solution to (\ref{eqn:chainhomotopyMovingBoundary2}) with a given $\theta$. Let us now describe the points in the zero-dimensional boundary. They are of one of several types:

  \begin{enumerate}
        \item Pairs $(v,w)$ where either:
        \begin{enumerate}
            \item $v$ is a Floer solution with asymptotes $\gamma_-$ and $\gamma_0$, \\for $CW^{\mu}({G}^\mu,\widehat{L} \to \widehat{L^\prime_0})$, and $w=(u,\theta)$ is a solution to Equation (\ref{eqn:chainhomotopyMovingBoundary2}), with asymptotes $\gamma_0$ and $\gamma_+$, for $\gamma_-,\gamma_0,\gamma_+$ some Hamiltonian chords.
            \item $v=(u,\theta)$ is a solution to Equation (\ref{eqn:chainhomotopyMovingBoundary2}), with asymptotes $\gamma_-$ and $\gamma_0$, and $w$ is a Floer solution with asymptotes $\gamma_0$ and $\gamma_+$, \\for $CW^{\mu}({G}^\mu,\widehat{L} \to \widehat{L^\prime_0})$. for $\gamma_-,\gamma_0,\gamma_+$ some Hamiltonian chords.
        \end{enumerate}
        \item Solutions $u$ solving Equation (\ref{eqn:continuationMovingBoundary}) with parameters corresponding to $\Psi$.
        \item Continuation solutions $u$, for the homotopy $H_{s+R}$. 
    \end{enumerate}
    Note that counting the solutions from the last item defines in homology the same map $\iota^{H^\mu, G^\mu}$, due to the independence of the choice of homotopy between $H^\mu$ and $G^\mu$.
    
 The boundary relation thus implies that $\mathfrak{H}$ is a chain homotopy, which proves that the diagram (\ref{eqn:commDiagSquareContRes2}) commutes in homology.

\begin{proposition}\label{prop:iotaFactorization}
    In the setting of Proposition \ref{prop:conicalCommDiagram}, for every $\mu \ge 0$, the morphism $\iota_{\mu \to \frac{1}{\varepsilon}\mu}$ factors through the morphism $\iota^{}_{\frac{1}{\varepsilon}(\mu-T)\to\frac{1}{\varepsilon}\mu}$, namely there exists a morphism $\psi$ such that the following diagram commutes:
    \begin{equation*}
        \begin{tikzcd}
        {HW^{\mu}\left(M,L\to L^\prime\right)} \arrow[dd, "\iota_{\mu \to \frac{1}{\varepsilon}\mu}"'] \arrow[rrd, "\psi"] &  &                                                                                                                                                    \\
                                                                                                                           &  & {HW^{\frac{1}{\varepsilon}(\mu-T)}\left(M,L\to L^\prime\right)} \arrow[lld, "\iota^{}_{\frac{1}{\varepsilon}(\mu-T)\to\frac{1}{\varepsilon}\mu}"'] \\
        {HW^{\frac{1}{\varepsilon}\mu}\left(M,L\to L^\prime\right)}                                                        &  &                                                                                                                                                   
        \end{tikzcd}
    \end{equation*}
\end{proposition}
\begin{proof}
    We claim that the Following diagram commutes:
    \begin{equation*}
        \begin{tikzcd}
        {HW^{\mu}\left(M,L\to L^\prime_0\right)} \arrow[dd, "{\mathcal{R}es}^{}_{M\to M_\varepsilon}"] \arrow[rr, "\mathcal{C}ont_{\{L^\prime_\tau\}_{0}^{T}}"] \arrow[ddd, "\iota_{\mu \to \frac{1}{\varepsilon}\mu}"', bend right=88] &  & {HW^{\mu-T}\left(M,L\to L^\prime_T\right)} \arrow[d, "{\mathcal{R}es}^{}_{M\to M_\varepsilon}"]                                                                                                    \\
                                                                                                                                                                                                                                                                             &  & {HW^{\mu-T}\left(M_\varepsilon,L\cap M_\varepsilon\to L^\prime_T\cap M_\varepsilon\right)}                                                                                                            \\
        {HW^{\mu}\left(M_\varepsilon,L\cap M_\varepsilon\to L^\prime_0\cap M_\varepsilon\right)} \arrow[d, "\rotatebox{90}{$\sim$}"]                                                                                                                                         &  & {HW^{\mu-T}\left(M_\varepsilon,L\cap M_\varepsilon\to L^\prime_0\cap M_\varepsilon\right)} \arrow[ll, "\iota^{}_{(\mu-T)\to\mu}"'] \arrow[d, "\rotatebox{90}{$\sim$}"] \arrow[u, equal] \\
        {HW^{\frac{1}{\varepsilon}\mu}\left(M,L\to L^\prime_0\right)}                                                                                                                                                                                                        &  & {HW^{\frac{1}{\varepsilon}(\mu-T)}\left(M,L\to L^\prime_0\right)} \arrow[ll, "\iota^{}_{\frac{1}{\varepsilon}(\mu-T)\to\frac{1}{\varepsilon}\mu}"'].                                                  
        \end{tikzcd}  
    \end{equation*}
    Indeed, the upper square commutes by Proposition \ref{prop:conicalCommDiagram}.
    The lower square commutes by the first item of Lemma \ref{lem:viterboRestriction}. The second item in that same lemma states that the composition of the two vertical morphisms on the left is indeed $\iota^{}_{\mu\to\frac{1}{\eps}\mu}$.
    
    The map $\psi$ is obtained as the composition of ${\mathcal{C}ont}_{\{L^\prime_\tau\}_{0}^{T}}$,  ${\mathcal{R}es}^{}_{M\to M_\varepsilon}$ and the isomorphism $\varphi_{\mu-T}$ which appears in the diagram as the unlabeled vertical isomorphism on the right.
    \[\psi := \varphi_{\mu-T} \circ {\mathcal{R}es}^{}_{M\to M_\varepsilon} \circ {\mathcal{C}ont}_{\{L^\prime_\tau\}_{0}^{T}}\]
\end{proof}
   \begin{figure}[H]
    	\centering
    	\includegraphics[scale=0.45]{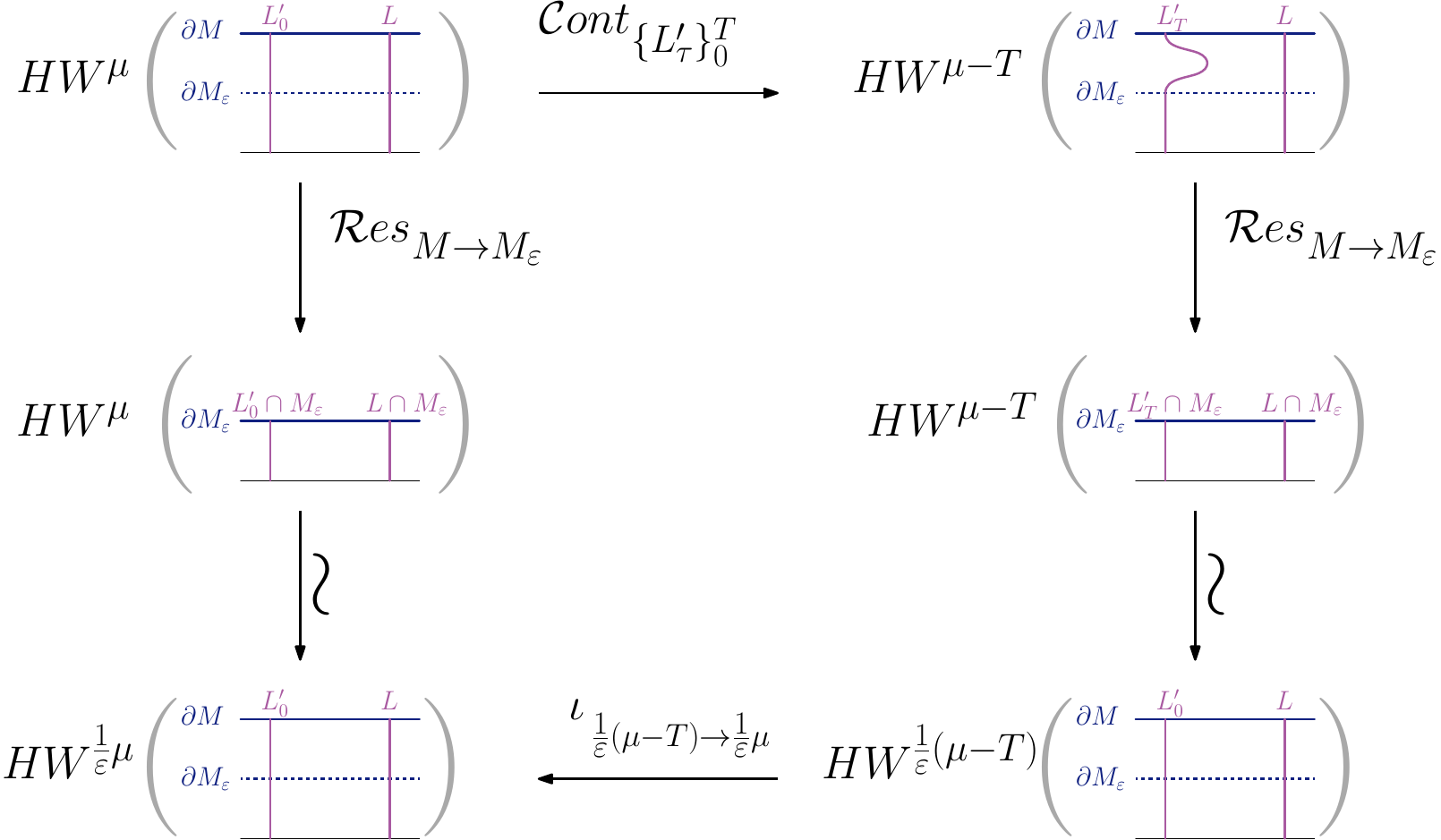}
    	\caption{\small{A pictorial representation of the diagram in the proof of Proposition \ref{prop:iotaFactorization}.
    	}}
    	\label{fig:proofDiagram}
    \end{figure}
\subsection{Proof of Theorem \ref{thm:specialCaseMainTheorem}, a Lower Bound on \texorpdfstring{$\operatorname{pb}^+$}{pb+}.}
We are finally ready to prove Theorem \ref{thm:specialCaseMainTheorem}.
\begin{proof}[Proof of Theorem \ref{thm:specialCaseMainTheorem}.]
    Suppose that the barcode of $\wfhBold{M}{L}{L^\prime}$ contains a bar of the form $\left (\mu, C\mu\right]$ with $\mu > 0$, $C>1$. with $\frac {1}{\varepsilon} \le C$.
    
    We first reduce to the case $\frac{1}{\eps} < C$, using the semi-continuity of $\operatorname{pb}^+$. The rest of the proof relies on the theory of filtered wrapped Floer homology. 

    \textbf{Reduction step:} Assuming the theorem has been proved for all $\varepsilon$ such that $\frac{1}{\eps} < C$, we will show how the result follows for $\eps$ such that $\frac{1}{\eps} = C$.
    Choose a sequence $\eps_n > \eps$ converging to $\eps$.
    Then the sets $\widehat{L}\cap M_{[\eps_n,1]}$, $\widehat{L^\prime}\cap M_{[\eps_n,1]}$, $\partial M_{\eps_n}$, and $\partial M$, converge in the Hausdorff distance to $\widehat{L}\cap M_{[\eps,1]}$, $\widehat{L^\prime}\cap M_{[\eps,1]}$, $\partial M_{\eps}$, and $\partial M$, respectively.
    Hence, by the semi-continuity of $\operatorname{pb}^+$:
    \begin{multline*}  
        \operatorname{pb}_{\widehat {M}}^+\left(\widehat{L}\cap M_{[\eps,1]},\widehat{L^\prime}\cap M_{[\eps,1]}, \partial M_\eps, \partial M \right) \ge \\
        \ge \limsup_{n\to\infty} \operatorname{pb}_{\widehat {M}}^+\left(\widehat{L}\cap M_{[\eps_n,1]},\widehat{L^\prime}\cap M_{[\eps_n,1]}, \partial M_{\eps_n}, \partial M \right) \ge \\
        \ge \lim_{n\to \infty} \frac {1}{\mu \left(1-\eps_n\right)} = \frac {1}{\mu \left(1-\eps\right)}.
    \end{multline*}

    \textbf{Floer theoretical step:}
    From this point to the end of the proof we assume that $\eps$ satisfies $\frac {1}{\varepsilon} < C$. Set $\nu_\delta = \mu+\delta$ for $\delta>0$ small enough so that $\frac {1}{\varepsilon}\left(\mu+\delta\right) < C$.
    Construct the Lagrangian isotopy from the admissible pairs $(F,G)$ as described in Section \ref{ssec:lagrangianIsotopy}.
    We apply Proposition \ref{prop:iotaFactorization} for $\nu_\delta$, obtaining the commuting diagram:
    \begin{equation*}
        \begin{tikzcd}
        {HW^{\mu+\delta}\left(M,L\to L^\prime\right)} \arrow[dd, "\iota_{\left(\mu + \delta\right) \to \frac{1}{\varepsilon}\left(\mu + \delta\right)}"'] \arrow[rrd, "\psi"] &  &                                                                                                                                                                                        \\
                                                                                                                                                                              &  & {HW^{\frac{1}{\varepsilon}(\mu+\delta-T)}\left(M,L\to L^\prime\right)} \arrow[lld, "\iota^{}_{\frac{1}{\varepsilon}(\mu+\delta-T)\to\frac{1}{\varepsilon}\left(\mu + \delta\right)}"'] \\
        {HW^{\frac{1}{\varepsilon}\left(\mu+\delta\right)}\left(M,L\to L^\prime\right)}                                                                                       &  &     .                                                                                                                                                                                  
        \end{tikzcd}
    \end{equation*}
    Note that the way the primitives of the Liouville form evolve when restricted to $L^\prime_\tau$ is described in Proposition \ref{prop:lagrangianIsotopyProperties}, and this is what yields the shift by $-T$ of the action filtration by the map $\psi$.
    Since the barcode of $\wfhBold{M}{L}{L^\prime}$ contains a bar of the form $\left (\mu, C\mu\right]$, it follows from Lemma \ref{lem:pmNonZeroElement} in Section \ref{sec:perModules} that there exists an element $x\in HW^{\mu+\delta}\left(M,L\to L^\prime\right)$ such that the following hold:
    \begin{itemize}
        \item $\iota_{\left(\mu + \delta\right) \to s}\,x \neq 0$ for all $\mu + \delta\le s\le C\mu$.
        \item For all $\mu+\delta \le s \le C\mu$ and $\sigma\le \mu$, 
        \[\iota_{\left(\mu+\delta\right)\to s}\,x \not\in \operatorname{im}\iota_{\sigma\to s}.\]
    \end{itemize}

    Pick such $x$. By the commutative diagram one has 
    \[\iota_{\left(\mu+\delta\right)\to \frac{1}{\eps}\left(\mu + \delta\right)}\,x = \iota_{\frac{1}{\eps}\left(\mu + \delta - T\right) \to \frac{1}{\eps}\left(\mu + \delta\right)}\,\psi(x).\]
    Setting $s=\frac{1}{\eps}(\mu + \delta)$ and $\sigma = \frac{1}{\eps}(\mu +\delta - T)$, in the second item above, it follows that:
    \begin{equation}\label{eqn:muInequality}
        \mu < \frac{1}{\eps}\left(\mu + \delta - T\right),
    \end{equation}
    hence
    \[
    \eps\mu < \mu + \delta - T,
    \]
    so
    \[
    T < \left(1-\eps\right)\mu + \delta.
    \]
    Since the above equation holds for all $\delta>0$ we obtain
    \[
    T \le \left(1-\eps\right)\mu.
    \]
    An intutaive explanation for Equation \ref{eqn:muInequality} is illustrated in figure \ref{fig:muInequality}.
    \begin{figure}[H]
    	\centering
    	\includegraphics[scale=0.75]{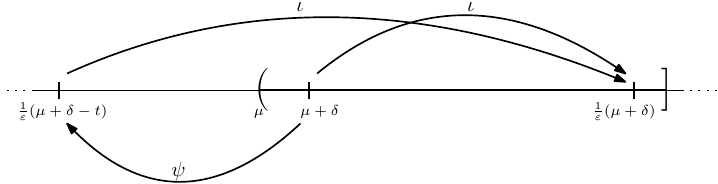}
    	\caption{\small{The map $\psi$ cannot map to the left of $\mu$ or we would have a factorizaiton of a non-zero morphism via the zero morphism.
    	}}
    	\label{fig:muInequality}
    \end{figure}
    
    We have obtained a bound on the interval length of induced deformations of the symplectic form, and therefore by the equality in Claim \ref{clm:pbConvenientForm} and by Lemma \ref{lem:pb4DeformationBound} we obtain
    \begin{multline*}
        \operatorname{pb}_{\widehat {M}}^+\left(L \cap  M_{[\varepsilon,1]}, L^\prime \cap  M_{[\varepsilon,1]}, \partial M_\eps, \partial M \right) = \\
        =
        \operatorname{pb}^+_{\widehat{M}} \left(\widehat{L}, \widehat{L^\prime},  \partial M_\varepsilon, M_{[1,\infty)} \right) \ge \frac{1}{\left(1-\eps\right)\mu},
    \end{multline*}
    which is the required bound.
\end{proof}

\section{Cotangent Bundles and the Proof of Theorem  \ref{thm:cotangentInterlinking}}\label{sec:ctgtBundles}
In this section we prove the bounds on $\operatorname{pb^+}$ of configurations involving cotangent fibers in cotangent bundles. Bounds from which Theorem \ref{thm:cotangentInterlinking} follows. The proof is based on a computation of wrapped Floer homology for two cotangent fibers. We prove:

\begin{theorem}\label{thm:cotBundleFibers}
Let $(N,g)$ be a closed Riemannian manifold. Denote by $S_r^*N$ the cosphere bundle of radius $r$ inside $T^*N$. Pick $x,y\in N$, two points of distance $d$. Then, denoting by $T^*_q N$ the cotangent fiber over the correspoinding point $q$, we have that for all $0<a<b$:
\begin{enumerate}
    \item 
    \hfill $\displaystyle
    \operatorname{pb}_{T^*\!N}^+\left(T^*_x N,\, T^*_y N,\, S^*_a N,\, S^*_b N\right) \ge \frac{1}{d\left(b-a\right)}. \hfill    $
    \item \hfill $\displaystyle
    \operatorname{pb}_{T^*\!N}^+\left(T^*_x N,\, T^*_y N,\, 0^{}_{T^*\!N},\, S^*_a N\right) \ge \frac{1}{da}, \hfill$
    
    where $0^{}_{T^*\!N}$ denotes the zero section in $T^*\!N$.
\end{enumerate}
In particular, 
\begin{enumerate}
    \item 
    The pair $\left(S^*_a N, S^*_b N\right)$ autonomously $d\left(b-a\right)$-interlinks the pair $\left(T^*_x N, T^*_y N\right)$, and vice versa.
    \item 
    The pair $\left(0^{}_{T^*\!N}, S^*_a\right)$ autonomously $da$-interlinks the pair $\left(T^*_x N, T^*_y N\right)$, and vice versa.
\end{enumerate}
\end{theorem}

In \cite{cieliebak2023loop} they construct a filtered isomorphism between between Morse theory on the loop space and symplectic homology of the unit coball bundle of $N$. In what follows, we describe an adaptation of \cite{cieliebak2023loop} to the setting of path spaces and Wrapped Floer homology of two cotangent fibers. The proof for the Lagrangian case is essentially the same as in \cite{cieliebak2023loop}, and we refer the reader to Section 5 in \cite{cieliebak2023loop} for the proof in the closed string setting.

We start with a brief review of Morse theory of the energy functional on a Riemannian manifold, we refer to \cite{abbondandolo2006floer} and the references therein for more details and proofs.

Given a Riemanninan manifold $(N,g)$, for every two points $x,y\in N$ we define the following space of paths from $p$ to $q$ in Sobolev class $W^{1,2}$:
\[
\mathcal{P}_{x\to y} := \left\lbrace \gamma \in W^{1,2}\left([0,1],N\right) \, \middle\vert\, \gamma(0)=x, \gamma(1)=y \right \rbrace.
\]
This is a Hilbert manifold.

The energy of a path $\gamma \in \mathcal{P}_{x\to y}$ is defined to be:
\[
E(\gamma) := \intop_0^1 \Vert \dot{\gamma(t)} \Vert^2_g dt\,.
\]
The critical points of the energy functional are geodesics of $g$ going from $x$ to $y$, parametrized proportionally to arc length.

If $x$ and $y$ are non-conjugate points, then $E$ is Morse, that is, the critical points of $E$ are non-degenerate. Fixing $x$, almost every $y\in N$ is non-conjugate to $x$. See, e.g \cite[Theorem 14.1 and Corollary 18.2]{milnor1963morse}. In what follows we assume non conjugacy in our computations, and we will later obtain the bound on $pb^+$ for any two general points via approximation and semi-continuity.

\begin{description}
    \item[Assumption:] $x$ and $y$ are non-conjugate.
\end{description}

Denote by $\hmBold{\mathcal{P}_{x\to y}}{\sqrt{E}}$ the persistence module given by the homology of the Morse complex of $E$ on $\mathcal{P}_{x\to y}$, filtered by the values of $\sqrt{E}$. (Since the Morse differential decreases energy, the homology is filtered by $E$ and hence also by $\sqrt{E}$, as square root is a monotone function.)
See \cite{abbondandolo2006floer} for the construction of Morse complex on the free loop space and on the based path spaces. See \cite{cieliebak2023loop} for the discussion on the filtration.

Consider the cotangent bundle of $N$ with the standard Liouville form and the induced symplectic structure, namely $\omega_{\operatorname{can}} = d\lambda_{\operatorname{can}} = d\left(\Sigma pdq\right)$.

Let $D_1^*N$ denote the unit coball bundle of $N$, namely 
\[
D_1^*N = \left\lbrace \left(p,q\right) \,\middle\vert\, q\in N,\, p\in T_q^*N,\, \Vert p \Vert^{}_g \le 1 \right\rbrace.
\]
This is a Liouville domain whose completion is symplectomorphic to $T^*N$.
The cotangent fibers $T^*_x N$ are exact conical Lagrangians, with $\lambda_{\operatorname{can}}\vert^{}_{T^*_x N} = 0 = d0$, and so, we pick the zero function $0\colon T^*_x N \to \mathbb R$ as the primitive of $\lambda_{\operatorname{can}}\vert^{}_{T^*_x N}$.

In \cite{cieliebak2023loop} an isomorphism of persistence modules is constructed between the symplectic cohomology of $D_1^*N$ filtered by the action functional and the persistence module given by the Morse complex on free loops, with filtration given by $\sqrt{E}$, extending the work of Viterbo \cite{viterbo2018functors} and of Abbondandollo-Schwartz \cite{abbondandolo2006floer}. An analogous proof works for the case of path spaces, providing the following theorem:

\begin{theorem}[Path space analogue of {\cite[Theorem 5.3]{cieliebak2023loop}}]\label{thm:wfhAndMorse}
There exists an isomorphism of persistence modules:
\[ 
    \Psi \colon \wfhBold{D_1^*N}{T^*_xN}{T^*_yN} \to \hmBold{\mathcal{P}_{x\to y}}{\sqrt{E}}.
\]
Namely, for every $\mu<\nu\in \mathbb{R}$ there exist isomorphisms of modules $\Psi_\mu, \Psi_\nu$ such that the following diagram commutes:
\begin{equation*}
    %
    \begin{tikzcd}
    {\operatorname{HW}^\mu\left(D_1^*N, T^*_xN \to T^*_yN\right) } \arrow[d, "\Psi_\mu"] \arrow[rr, "\iota^{}_{\mu\to\nu}"] &  & {\operatorname{HW}^\nu\left(D_1^*N, T^*_xN \to T^*_yN\right) } \arrow[d, "\Psi_\nu"] \\
    {\operatorname{HM}^\mu\left(\mathcal{P}_{x\to y}\,; \sqrt{E}\right)} \arrow[rr, "\iota^{}_{\mu\to\nu}"]                   &  & {\operatorname{HM}^\nu\left(\mathcal{P}_{x\to y}\,; \sqrt{E}\right)}.                 
    \end{tikzcd}
\end{equation*}
\end{theorem}

\subsection{Proof of Theorem \ref{thm:cotBundleFibers}}
We first reduce to the case where $y$ is non-conjugate to $x$. By genericty, choose a sequence of points $y_n$ converging to $y$, such that every $y_n$ is non-conjugate to $x$. We obtain $\operatorname{d}(x,y_n) \to \operatorname{d}(x,y)$, and that the cotangent fibers converge in the Hausdorff distance, $T^*_{y_n}N \to T^*_{y}N$. Hence by the semi-continuity property of $\operatorname{pb}^+$:
\begin{multline*}  
    \operatorname{pb}_{T^*N}^+\left( T^*_x N,\, T^*_y N,\, S^*_a N,\, S^*_b N\right) \ge \\
    \ge \limsup_{n\to\infty} \operatorname{pb}_{T^*N}^+\left( T^*_{x_n} N,\, T^*_{y_n} N,\, S^*_a N,\, S^*_b N\right)  \ge \\
    \ge \lim_{n\to \infty} \frac {1}{\operatorname{d}\left(x,y_n\right)\cdot\left(b-a\right)} = \frac {1}{\operatorname{d}\left(x,y\right)\cdot\left(b-a\right)}.
\end{multline*}


Assuming $x$ and $y$ are non-conjugate, we will show that the barcode of $\wfhBoldSmallPar{D^*_1 N}{T^*_x}{T^*_y}$ has a bar of the form $\left(d,\infty\right)$, with $d=d\left(x,y\right)$. The result will then follow from Theorem $\ref{thm:mainTheorem}$ in the case of a semi-inifine bar.

Recall that by Theorem \ref{thm:wfhAndMorse} there is in isomorphism of persistence modules between $\wfhBoldSmallPar{D^*_1 N}{T^*_x}{T^*_y}$ and the Morse homology of the energy functional on the path space from $x$ to $y$, filtered by the square roof of the energy, namely $\hmBoldSmallPar{\mathcal{P}_{x\to y}}{\sqrt{E}}$.
Since the metric is complete, there exists a geodesic $\gamma \colon [0,1]\to \mathbb{R}$ from $x$ to $y$, realizing the distance between $x$ and $y$, that is, $\operatorname{length}(\gamma) = \operatorname{d}(x,y)$. Since $\gamma$ is a global minimum of $E$, it is of Morse index $0$.
Moreover, since $\gamma$ is parametrized proportionally to the arc length, $E(\gamma) = \operatorname{length}(\gamma)^2$. Thus, in any sublevel of $\sqrt{E}$ above the sublevel set $\lbrace \sqrt{E} < \operatorname{d}(x,y)\rbrace$, the geodesic $\gamma$ represents a closed chain in the Morse complex (since it is of index $0$ and no critical point is of index $-1$) and it is never exact, since it is a global minimum and represents the point class in the zero degree homology of the path space. Therefore it is a generator of a bar of the form $\left(d,\infty\right)$, proving that such a bar exists in the barcode. We thus obtain the desired bounds:
\begin{align*}
    &\operatorname{pb}_{T^*N}^+\left(T^*_x N,\, T^*_y N,\, S^*_a N,\, S^*_b N\right) \ge \frac{1}{d\cdot \left(b-a\right)}\text{,} \\
    \text{and } &\operatorname{pb}_{T^*N}^+\left(T^*_x N,\, T^*_y N,\, 0^{}_{T^*N},\, S^*_a N\right) \ge \frac{1}{d\cdot a}.
\end{align*}
Note that $x$ and $y$ in the theorem are completely interchangeable, since the distance function is symmetric $d(x,y)=d(y,x)$. Hence also the same bounds apply to 
\begin{align*}
    &\operatorname{pb}_{T^*N}^+\left(T^*_y N,\, T^*_x N,\, S^*_a N,\, S^*_b N\right) \ge \frac{1}{d\cdot \left(b-a\right)}, \\
    \text{and } &\operatorname{pb}_{T^*N}^+\left(T^*_y N,\, T^*_x N,\, 0^{}_{T^*N},\, S^*_a N\right) \ge \frac{1}{d\cdot a}.
\end{align*}

The interlinking result follows from Theorem \ref{thm:pb4AndInterlinking}, together with the anti-symmetry of $\operatorname{pb}^+$. 

\section{Discussion}\label{sec:discussion}
Several aspects of the connection between interlinking and Floer theory remain unexplored. Here, we highlight a few key points for further investigation.
In \cite{entov2022legendrian} interlinking with respect to non-autonomous Hamiltonians is also discussed, and is obtained via stabilization, namely, bounds on $\operatorname{pb}^+$ in $M\times T^*S^1$, where each set in in quadruple in $M$ is replaced with its product with the zero section.
Moreover it is conjectured there that the existence of a bar of the form $(a,\infty)$ in the barcode, should imply the stronger, non-autonomous interlinking.
In the setting layed out by our paper, one might approach this via a suitable filtered K\"unneth isomorphism for wrapped Floer homology of Lagrangian submanifolds. We intend to pursue this direction in a future work.

Our result for cotangent fibers and cosphere bundles in a cotangent bundle was based on a computation reducing wrapped Floer homology to Morse theory of another functional, in our case, the energy functional on paths in a Riemanninan manifold. Another case where Floer theory can be reduced to Morse theory is that of convex domains.
In \cite{abbondandolo2022symplectic}, an isomorphism between symplectic homology of convex domains, and Morse theory of a Clarke's dual action functional is established. We believe it would be interesting to explore an open string analogue of this isomorphism, and perhaps obtain new computations of barcodes in wrapped Floer homology of Lagrangians in such domains, which in turn will provide new examples of interlinking.

Moreover, in discussions with Entov and Polterovich, they informed the author of certain invariants of morphisms of persistence modules, which, in their setting of Symplectic field theory, can be appllied to prove interlinking in non conical configurations. These invariants are to appear in a yet to be published work of theirs. We believe that these invariants can also be applied to our setting of wrapped Floer homology of Lagrangians, obtaining analogous results.

\bibliographystyle{alpha}

\bibliography{Bibliography}

\bigskip

\noindent
\begin{tabular}{l}
{\bf Yaniv Ganor} \\
School of Mathematical Sciences,\\
Faculty of Sciences,\\
HIT - Holon Institute of Technology \\

Golomb St. 52, Holon 5810201, Israel\\
{\em E-mail:}  \texttt{ganory@gmail.com}
\end{tabular}

\end{document}